\documentclass[runningheads,envcountsame]{llncs}
\usepackage{graphicx}
\usepackage{mathtools}
\DeclarePairedDelimiter{\group}{(}{)}
\DeclarePairedDelimiter{\sqgroup}{[}{]}
\DeclarePairedDelimiter{\set}{\{}{\}}
\DeclarePairedDelimiter{\avg}{\llbracket}{\rrbracket}

\DeclarePairedDelimiter{\norm}{\Vert}{\Vert}
\DeclarePairedDelimiter{\abs}{\vert}{\vert}

\newcommand{\xqed}[1]{%
  \leavevmode\unskip\penalty9999 \hbox{}\nobreak\hfill
  \quad\hbox{#1}}

\newcommand{\pths}{\Omega}
\newcommand{\posspace}{\mathcal{X}}

\newcommand{\sits}{\mathbb{S}}
\newcommand{\outcomes}{\set{0,1}}
\newcommand{\selections}{\set{0,1}}

\newcommand{\frcstsystem}{{\lex_\bullet}}

\newcommand{\frcstsystems}{{\smash{\lexs^\sits}}}
\newcommand{\frcstsystemscomp}{\smash{\lexs^\sits_{\mathrm{C}}(\pth)}}
\newcommand{\frcstsystemscompstat}{\lexs_{\mathrm{C}}(\pth)}

\newcommand{\process}{F}
\newcommand{\processtoo}{G}

\newcommand{\selection}{S}
\newcommand{\multprocess}{D}
\newcommand{\mint}[1][\multprocess]{#1^\circledcirc}
\newcommand{\supermartin}{M}
\newcommand{\submartin}{M}
\newcommand{\submartins}[1][\frcstsystem]{\underline{\mathbb{M}}(#1)}
\newcommand{\supermartins}[1][\frcstsystem]{\overline{\mathbb{M}}(#1)}

\newcommand{\test}{\supermartin}
\newcommand{\callowable}[1]{\overline{\mathbb{M}}_{\mathrm{C}}(#1)}
\newcommand{\callowabletests}[1]{\overline{\mathbb{M}}_{\mathrm{C}}(#1)}

\newcommand{\naturals}{\mathbb{N}}
\newcommand{\naturalswithzero}{\mathbb{N}_0}

\newcommand{\reals}{\mathbb{R}}

\newcommand{\rationals}{\mathbb{Q}}

\newcommand{\ex}{E}
\newcommand{\lex}{\underline{\ex}}
\newcommand{\lexs}{\underline{\mathcal{E}}}
\newcommand{\uex}{\overline{\ex}}

\newcommand{\gambleinterval}[1][]{\sqgroup{\underline{\gamble},\overline{\gamble}}}

\newcommand{\init}{\square}
\newcommand{\pth}{\omega}

\newcommand{\sit}{s}

\newcommand{\xval}[1][]{x_{#1}}
\newcommand{\xvaltolong}[1][n]{\xval[1],\dots,\xval[#1]}
\newcommand{\xvalto}[1][n]{\xval[1:#1]}
\newcommand{\xvaltok}{\xval[1:k]}
\newcommand{\xvaltom}{\xval[1:m]}
\newcommand{\xvaltol}{\xval[1:l]}

\newcommand{\xvalnextk}{\xval[k+1]}

\newcommand{\xvaltolongsum}[1][n]{\xval[1],\dots,\xval[#1]}

\newcommand{\xvalnextsum}{\xval[k+1]}

\newcommand{\gamble}{f}
\newcommand{\gambles}{\mathcal{L}(\posspace)}

\newcommand{\then}{\Rightarrow}

\newcommand{\ifandonlyif}{\Leftrightarrow}

\newcommand{\adddelta}{\Delta}

\newcommand{\varnorm}[1]{\norm{#1}_\mathrm{v}}

\usepackage{amsmath}

\usepackage{newtxmath}

\usepackage{nicefrac}

\usepackage{xcolor}

\usepackage{enumitem}

\usepackage[hidelinks]{hyperref}

\allowdisplaybreaks

\renewenvironment{proof}[1][\proofname]{{\noindent\bfseries #1 }}{\qed}\par\vspace{\topsep}

\usepackage{etoolbox}
\usepackage{verbatim}   

\newbool{ifnotSUM}
\setbool{ifnotSUM}{true}

\newenvironment{SUM}{}{}

\ifbool{ifnotSUM}{\AtBeginEnvironment{SUM}{\comment}%
\AtEndEnvironment{SUM}{\endcomment}}{}

\newbool{ifnotArxiveExt}
\setbool{ifnotArxiveExt}{false}

\newenvironment{ArxiveExt}{}{}

\ifbool{ifnotArxiveExt}{\AtBeginEnvironment{ArxiveExt}{\comment}%
\AtEndEnvironment{ArxiveExt}{\endcomment}}{}

\begin{document}
\title{Computable randomness is about more than probabilities
}
%
%
\author{Floris Persiau 
\and
Jasper De Bock 
\and
Gert de Cooman 
}
\authorrunning{F.~Persiau et al.}
%
\institute{FLip, ELIS, Ghent University, Belgium}
\maketitle              
\begin{abstract}
We introduce a notion of computable randomness for infinite sequences that generalises the classical version in two important ways. 
First, our definition of computable randomness is associated with imprecise probability models, in the sense that we consider lower expectations---or sets of probabilities---instead of classical `precise' probabilities. 
Secondly, instead of binary sequences, we consider sequences whose elements take values in some finite sample space. 
Interestingly, we find that every sequence is computably random with respect to at least one lower expectation, and that lower expectations that are more informative have fewer computably random sequences. 
This leads to the intriguing question whether every sequence is computably random with respect to a unique most informative lower expectation. 
We study this question in some detail and provide a partial answer.

\end{abstract}
\section{Introduction}
\label{sec:intro}

When do we consider an infinite sequence $\pth = (\xvaltolong,\dots)$, whose individual elements $x_n$ take values in some finite sample space $\posspace$, to be random? 
This is actually not a fair question, because randomness is never defined absolutely, but always relative to an uncertainty model. 
Consider for example an infinite sequence generated by repeatedly throwing a single fair die and writing down the number of eyes on each throw. 
In this case, we would be justified in calling this sequence random with respect to a precise probability model that assigns probability $\nicefrac{1}{6}$ to every possible outcome.

It is exactly such precise probability models that have received the most attention in the study of randomness \cite{ambosspies2000,bienvenu2009,Rute2012}. 
Early work focused on binary sequences and the law of large numbers that such sequences, and computably selected subsequences, were required to satisfy: an infinite binary sequence of zeros and ones is called \emph{Church random} if the relative frequencies in any computably selected subsequence converge to $\nicefrac{1}{2}$ \cite{ambosspies2000}. 
Schnorr, inspired by the work of Ville, strengthened this definition by introducing a notion of computable randomness \cite{Rute2012}. 
On his account, randomness is about betting. 
The starting point is that a precise probability model that assigns a (computable) probability $p$ to $1$ and $1-p$ to $0$ can be interpreted as stating that $p$ is a fair price for bet $\mathbb{I}_1(X_i)$ that yields $1$ when $X_i=1$ and $0$ when $X_i=0$, for every---a priori unknown---value $X_i$ of a binary sequence $\pth = (\xvaltolong,\dots)$ of zeros and ones. 
Such a sequence is then considered to be \emph{computably random} with respect to $p$ if there is no computable betting strategy for getting rich without bounds along $\pth$ without borrowing, simply by betting according to this fair price. 
Notably, binary sequences that are computably random for $p=\nicefrac{1}{2}$ are also Church random. 
So here too, the relative frequency of any element $x \in \posspace$ will converge to a limit frequency along $\pth$ --- $\nicefrac{1}{2}$ in the binary case for $p=\nicefrac{1}{2}$. 
In fact, this is typically true for any notion of randomness with respect to a precise probability model.

However, as has been argued extensively \cite{gorban2016}, there are various random phenomena where this stabilisation is not clearly present, or even clearly absent. 
Hence, only adopting precise probability models to define notions of random sequences is too much of an idealisation. 
Recently, this issue was addressed by De Cooman and De Bock for binary sequences by introducing a notion of computable randomness with respect to probability intervals instead of precise probability models, whose lower bounds represent supremum acceptable buying prices, and whose upper bounds represent infimum acceptable selling prices, again for the bet $\mathbb{I}_1(X_i)$ that, for every value $x_i$ of a binary sequence $\pth = (\xvaltolong,\dots)$, yields $1$ if $X_i=1$ and $0$ otherwise \cite{CoomanBock2017}. 

On this account, relative frequencies must not necessarily converge to a limit frequency along $\pth$, but may fluctuate within the probability interval. 

Here, we generalise the work done by De Cooman and De Bock \cite{CoomanBock2017} for binary sequences, and develop a similar concept for infinite sequences that take values in more general finite sample spaces. 
To this end, we consider an even more general framework for describing uncertainty: we use coherent lower expectations---or sets of probability mass functions---instead of probability intervals or probabilities. 
Loosely speaking, we say that an infinite sequence $\pth=(\xvaltolong,\dots)$ is \emph{computably random} with respect to a (forecasting system of) lower expectation(s), when there is no computable betting strategy for getting rich without bounds along $\pth$ without borrowing and by only engaging in bets whose (upper) expected profit is non-positive or negative.\footnote{A real number $x \in \reals$ is called positive if $x>0$, non-negative if $x \geq 0$, negative if $x <0$ and non-positive if $x \leq 0$.}

This contribution is structured as follows. 
We start in Section 2 with a brief introduction to coherent lower expectations, and explain in particular their connection with probabilities and their interpretation in terms of gambles and betting. 
Next, in Section 3, we define a subject's uncertainty for an infinite sequence of variables $X_1,\dots,X_n,\dots$ by introducing forecasting systems that associate with every finite sequence $(\xvaltolong)$ a coherent lower expectation for the variable $X_{n+1}$. 
This allows us to introduce corresponding betting strategies to bet on the infinite sequence of variables along a sequence $\pth=(\xvaltolong,\dots)$ in terms of non-negative (strict) supermartingales. 
After explaining in Section 4 when such a non-negative (strict) supermartingale is computable, we extend the existing notion of computable randomness from precise and interval probability models to coherent lower expectations in Section 5, and study its properties. 
The remainder of the paper focuses on special cases. 
When we restrict our attention to stationary forecasting systems that forecast a single coherent lower expectation in Section 6, it turns out that every sequence $\pth$ is computably random with respect to at least one coherent lower expectation and that if $\pth$ is computably random for some coherent lower expectation, then it is also computably random for any coherent lower expectation that is less informative, i.e., provides fewer gambles. 
This makes us question whether there is a unique most informative coherent lower expectation for which $\pth$ is computably random. 
After inspecting some examples, it turns out that such a most informative coherent lower expectation sometimes exists, but sometimes does not. 
When it does not, our examples lead us to conjecture that it `almost' exists. 
We conclude the discussion in Section 7 by introducing a derived notion of computable randomness with respect to a gamble $\gamble$ and an interval $I$ by focusing on the behaviour of coherent lower expectations on a specific gamble $\gamble$ of their domain. 
It turns out that for every gamble $\gamble$, a sequence $\pth$ is `almost' computably random with respect to some smallest interval.
\begin{SUM}
To adhere to the page constraints, all proofs are omitted. 
They are available in an extended on-line version~\cite{sum2020persiau:arxiv}.
\end{SUM}

\section{Coherent lower expectations}
\label{sec:singleforecast}

To get the discussion started, we consider a single uncertain variable $X$ that takes values in some finite set $\posspace$, called the sample space. 
A subject's uncertainty about the unknown value of $X$ can then be modelled in several ways. 
We will do so by means of a coherent lower expectation: a functional that associates a real number with every gamble, where a gamble $\gamble:\posspace \to \reals$ is a map from the sample space $\posspace$ to the real numbers. 
We denote the linear space of all gambles by $\gambles$. 

\begin{definition}
A coherent lower expectation $\lex \colon \gambles \to \reals$ is a real-valued functional on $\gambles$ that satisfies the following axioms. 
For all gambles $f,g \in \gambles$ and all non-negative $\alpha\in\reals$:
\begin{enumerate}[label=\upshape{C}\arabic*.,ref=\upshape{C}\arabic*,leftmargin=*]
\item\label{axiom:coherence:bounds} $\min\gamble \leq \lex(\gamble)$  \hfill{\upshape[boundedness]}
\item\label{axiom:coherence:homogeneity} $\lex(\alpha \gamble) = \alpha \lex(\gamble)$ \hfill{\upshape[non-negative homogeneity]}
\item\label{axiom:coherence:subadditivity} $\lex(\gamble) + \lex(g) \leq \lex(\gamble+g)$ \hfill{\upshape[superadditivity]}
\end{enumerate}
We will use \smash{$\lexs$} to denote the set of all coherent lower expectations on $\gambles$.
\end{definition}

As a limit case, for any probability mass function $p$ on $\posspace$, it is easy to check that the linear expectation $\ex_p$, defined by $\ex_p(\gamble) \coloneq \sum_{x \in \posspace}\gamble(x) p(x)$ for all $\gamble \in \gambles$, is a coherent lower expectation, which corresponds to a maximally informative or least conservative model for a subject's uncertainty. 
More generally, a coherent lower expectation $\lex$ can be interpreted as a lower envelope of such linear expectations. 
That is, there is always a (closed and convex) set~$\mathcal{M}$ of probability mass functions such that $\lex(\gamble)=\min\{\ex_p(\gamble)\colon p \in \mathcal{M}\}$ for all $\gamble \in \gambles$ \cite{walley1991}. 
In that sense, coherent lower expectations can be regarded as a generalisation of probabilities to (closed and convex) sets of probabilities.
Alternatively, the lower expectation $\lex(\gamble)$ can be interpreted directly as a subject's supremum buying price for the uncertain reward $\gamble$. 

The particular interpretation that is adopted is not important for what we intend to do here. 
For our purposes, the only thing we will assume is that when a subject specifies a coherent lower expectation, every gamble $\gamble \in \gambles$ such that $\lex(\gamble) >0$ is desirable to him and every gamble $\gamble \in \gambles$ such that $\lex(\gamble) \geq 0$ is acceptable to him. 
We think this makes sense under both of the aforementioned interpretations. 
Furthermore, as we will see in Section~\ref{sec:randomsequences}, the distinction between desirable and acceptable gambles does not matter for our definition of computable randomness. 
For now, however, we proceed with both notions. 

Whenever a subject specifies a coherent lower expectation, we can consider an opponent that takes this subject up on a gamble $\gamble$ on the unknown outcome $X$ in a betting game. 
Borrowing terminology from the field of game-theoretic probabilities \cite{ShaferVovk2019}, we will refer to our subject as Forecaster and to his opponent as Sceptic.
Forecaster will only bet according to those gambles $\gamble \in \gambles$ that are acceptable to him ($\lex(\gamble) \geq 0$), or alternatively, those that are desirable to him ($\lex(\gamble) > 0$). 
This leads to an unknown reward $\gamble(X)$ for Forecaster and an unknown reward $-\gamble(X)$ for Sceptic. 
After Sceptic selects such a gamble, the outcome $x \in \posspace$ is revealed, Forecaster receives the (possibly negative) reward $\gamble(x)$, and Sceptic receives the reward $-\gamble(x)$.
Equivalently, when considering for any coherent lower expectation $\lex$ the conjugate upper expectation $\smash{\uex}$, defined as \smash{$\uex(\gamble)\coloneq-\lex(-\gamble)$} for all $\gamble \in \gambles$, then Sceptic is allowed to bet according to any gamble $f \in \gambles$ for which $\uex(f) \leq 0$ (or $\uex(f) < 0$), leading to an uncertain reward $f(X)$ for Sceptic and an uncertain reward $-f(X)$ for Forecaster.
In what follows, we will typically take the perspective of Sceptic. 
The gambles that are available to her will thus be the gambles $f \in \gambles$ with non-positive (or negative) upper expectation $\uex(f)\leq 0$ ($\uex(f) < 0$).

An important special case is the so-called vacuous coherent lower expectation~$\lex_v$, defined by $\lex_v(\gamble)\coloneq\min \gamble$ for all $\gamble \in \gambles$. 
If Forecaster specifies~$\lex_v$, this corresponds to a very conservative attitude where he is only interested in gambles $\gamble$ that give him a guaranteed non-negative (or positive) gain, i.e., $\min \gamble \geq 0$ ($\min \gamble > 0$), implying that Sceptic has a guaranteed non-negative (or positive) loss, i.e., $\max \gamble \leq 0$ ($\max \gamble < 0$). 

\begin{example} \label{ExampleSection2}
Consider an experiment with three possible outcomes $A, B$ and $C$, i.e., $\posspace\coloneq\{A,B,C\}$, and three probability mass functions $p_0,p_1$ and $p_2$ defined by $(p_0(A),p_0(B),p_0(C))\coloneq(0,\nicefrac{1}{2},\nicefrac{1}{2})$, $(p_1(A),p_1(B),p_1(C))\coloneq(\nicefrac{1}{2},0,\nicefrac{1}{2})$ and $(p_2(A),p_2(B),p_2(C))\coloneq(\nicefrac{1}{2},\nicefrac{1}{2},0)$. 
We can then define a coherent lower expectation $\lex$ by $\lex(\gamble) \coloneq \min\{\ex_{p_0}(\gamble),\ex_{p_1}(\gamble),\ex_{p_2}(\gamble)\}$ for every gamble $\gamble \in \gambles$.
For the particular gamble $\gamble$ defined by $(f(A),f(B),f(C)) \coloneq (1,-2,3)$, the value of this lower expectation then equals $\lex(f)=\min\{\nicefrac{1}{2},2,\nicefrac{-1}{2}\}=\nicefrac{-1}{2}$.
\xqed{$\lozenge$}
\end{example}
\newpage

\section{Forecasting systems and betting strategies}
\label{sec:forecastingsystems}

We now consider a sequential version of the betting game in Section~\ref{sec:singleforecast} between Forecaster and Sceptic, by considering a sequence of variables $\allowbreak X_1,\dots, X_n \allowbreak , \dots$, all of which take values in our finite sample space $\posspace$. 

On each round of the game, indexed by $n \in \naturalswithzero\coloneqq\naturals\cup\{0\}$, 
the a priori unknown finite sequence of outcomes $x_{1:n}=(x_1,\dots,x_{n})$ has been revealed and we assume that Forecaster's uncertainty about the next---as yet unknown---outcome $X_{n+1} \in \posspace$ is described by a coherent lower expectation. 
Hence, on each round of the game, Forecaster's uncertainty can depend on and be indexed by the past states. 

All finite sequences $\sit =x_{1:n}=(\xvaltolong)$---so-called situations---are collected in the set $\sits \coloneqq \posspace^*=\bigcup_{n\in\naturalswithzero}\posspace^n$. 
By convention, we call the empty sequence the initial situation and denote it by $\init$.
The finite sequences $\sit \in \sits$ form an event tree, and it is on this whole event tree that we will describe Forecaster's uncertainty, using a so-called forecasting system.
\begin{definition}
\label{def:ForecastingSystem}
A \emph{forecasting system} $\frcstsystem \colon \sits \to \lexs$ is a map that associates with every situation $\sit \in \sits$ a coherent lower expectation $\lex_\sit \in \lexs$. 
The collection of all forecasting systems is denoted by~$\frcstsystems$. 
\end{definition}

Every forecasting system corresponds to a collection of bets that are available to Sceptic. 
That is, in every situation $\sit=\xvalto$, Sceptic is allowed to bet on the unknown outcome $X_{n+1}$ according to any gamble $\gamble \in \gambles$ such that $\uex_\sit(\gamble) \leq 0$ (or $\uex_\sit(\gamble) < 0$). 
This leads to an uncertain reward $\gamble(X_{n+1})$ for Sceptic and an uncertain reward $-\gamble(X_{n+1})$ for Forecaster. 
Afterwards, when the outcome $x_{n+1}$ is revealed, Sceptic gets the amount $\gamble(x_{n+1})$, Forecaster gets the amount $-\gamble(x_{n+1})$ and we move to the next round. 
To formalise this sequential betting game, we introduce the notion of a supermartingale, which is a special case of a so-called real process.

A real process $\process \colon \sits \to \reals$ is a map that associates with every situation $\sit=\xvalto \in \sits$ of the event tree, a real number $\process(\sit)$. 
With every real process $\process$ there corresponds a process difference $\adddelta \process$ that associates with every situation $\sit \in \sits$ a gamble $\adddelta \process (\sit)$ on $\posspace$, defined as $\adddelta \process (s)(x)\coloneqq\process(sx)-\process(s)$ for every $\sit \in \sits$ and $x \in \posspace$, where $\sit x$ denotes the concatenation of $\sit$ and $x$. 
We call a real process $\submartin$ a (strict) supermartingale if $\uex_s(\adddelta\supermartin(s)) \leq 0$ ($\uex_s(\adddelta\supermartin(s)) < 0$) for every situation $\sit \in \sits$. 
Note that a supermartingale is always defined relative to a forecasting system~$\frcstsystem$. 
Similarly, a real process $\submartin$ is called a (strict) submartingale if $\lex_s(\adddelta\submartin(s)) \geq 0$ ($\lex_s(\adddelta\submartin(s)) > 0$) for every $\sit \in \sits$. 
Due to the conjugacy relation between upper and lower expectations, $\supermartin$ is a (strict) supermartingale if and only if $-\supermartin$ is a (strict) submartingale. 
We collect the super- and submartingales in the sets $\supermartins$ and $\submartins$, respectively.
A supermartingale $\supermartin$ is called non-negative (positive) if $\supermartin(s)\geq 0$ ($\supermartin(s)> 0$) for all $\sit \in \sits$.

From the previous discussion, it is clear that Sceptic's allowable betting behaviour corresponds to supermartingales or strict supermartingales, depending on whether we consider acceptable or desirable gambles, respectively.
Indeed, in each situation $\sit=x_{1:n} \in \sits$, she can only select a gamble $\adddelta\supermartin(\sit)$ for which $\uex_\sit(\adddelta\supermartin(\sit)) \leq 0$ ($\uex_\sit(\adddelta\supermartin(\sit)) < 0$) and her accumulated capital $\supermartin(\xvalto) = \supermartin(\init)+\sum_{k=0}^{n-1}\adddelta\supermartin(\xvaltok)(\xvalnextk)$, with $\supermartin(\init)$ being her initial capital, will therefore evolve as a (strict) supermartingale. 
As mentioned before, it will turn out not to matter whether we consider acceptable or desirable gambles, or equivalently, supermartingales or strict supermartingales. 
To be able to explain why that is, we will proceed with both. 
In particular, we will restrict Sceptic's allowed betting strategies to non-negative (strict) supermartingales, where the non-negativity is imposed to prevent her from borrowing money. 
Non-negative supermartingales~$\test$ that start with unit capital~$\test(\init)$ are called test supermartingales.

\begin{example}
Consider a repetition of the experiment in Example~\ref{ExampleSection2}, and a stationary forecasting system $\frcstsystem$ defined by $\lex_\sit(\gamble)=\lex(\gamble)= \min\{\ex_{p_0}(\gamble),\ex_{p_1}(\gamble),\allowbreak \ex_{p_2}(\gamble)\}$ for every $\sit \in \sits$ and $\gamble \in \gambles$, with $p_0,p_1$ and $p_2$ as in Example~\ref{ExampleSection2}.
An example of a non-negative (test) supermartingale $\supermartin$ is then given by the recursion equation $\adddelta \test(\sit) = (\adddelta \test(\sit)(A),\adddelta \test(\sit)(B),\allowbreak\adddelta \test(\sit)(C)) \coloneq (\nicefrac{-\test(\sit)}{2},\nicefrac{\test(\sit)}{2},\nicefrac{-\test(\sit)}{2})$ for every $\sit \in \sits$, with $\test(\init) \coloneq 1$. E.g., for $\sit=A$, it follows that $\test(A)=M(\init)+\adddelta\test(\init)(A)=\test(\init)-\nicefrac{\test(\init)}{2}=\nicefrac{\test(\init)}{2}=\nicefrac{1}{2}$.
It is easy to see that $\supermartin$ is non-negative by construction and, for every $\sit \in \sits$, it holds that $\uex_\sit(\adddelta \test(\sit))=\max\{0,\nicefrac{-\test(\sit)}{2},0\}=0$.
\xqed{$\lozenge$}
\end{example}

In what follows, we will use Sceptic's allowed betting strategies---so non-negative (strict) supermartingales---to introduce a notion of computable randomness with respect to a forecasting system. 
We denote the set of all infinite sequences of states---or so-called paths---by $\pths \coloneqq \posspace^\naturals$ and, for every such path $\pth=(\xvaltolong,\dots)\in\pths$, we let $\pth^n\coloneqq(\xvaltolong)$ for all $n \in \naturalswithzero$.

However, not all betting strategies within the uncountable infinite set of all allowed betting strategies are implementable. 
We will therefore restrict our attention to those betting strategies that are computable, as an idealisation of the ones that can be practically implemented.

\section{A brief introduction to computability}
\label{sec:computability}

Computability deals with the ability to compute mathematical objects in an effective manner, which means that they can be approximated to arbitrary precision in a finite number of steps. 
In order to formalise this notion, computability theory uses so-called recursive functions as its basic building blocks \cite{LiVitanyi2008,Pour-ElRichards2016}.

A function $\phi\colon\naturalswithzero\to\naturalswithzero$ is recursive if it can be computed by a Turing machine, which is a mathematical model of computation that defines an abstract machine. 
By the Church--Turing thesis, this is equivalent to the existence of an algorithm that, upon the input of a natural number $n \in \naturalswithzero$, outputs the natural number $\phi(n)$. 
The domain $\naturalswithzero$ can also be replaced by any other countable set whose elements can be expressed by adopting a finite alphabet, which for example allows us to consider recursive functions from $\sits$ to $\naturalswithzero$ or from $\sits \times \naturalswithzero$ to $\naturalswithzero$. 
Any set of recursive functions is countable, because the set of all algorithms, which are finite sequences of computer-implementable instructions, is countable.

We can also consider recursive sequences of rationals, recursive rational processes and recursive nets of rationals.
A sequence $\{r_n\}_{n \in \naturalswithzero}$ of rational numbers is called recursive if there are three recursive maps $a,b,\sigma$ from $\naturalswithzero$ to $\naturalswithzero$ such that $b(n)\neq 0$ for all $n \in \naturalswithzero$ and \smash{$r_n=(-1)^{\sigma(n)}\frac{a(n)}{b(n)}$} for all $n \in \naturalswithzero$.
By replacing the domain $\naturalswithzero$ with $\sits$, we obtain a recursive rational process.
That is, a rational process $\process\colon \sits \to \rationals$ is called recursive if there are three recursive maps $a,b,\sigma$ from $\sits$ to $\naturalswithzero$ such that $b(\sit)\neq 0$ for all $\sit \in \sits$ and \smash{$\process(\sit)=(-1)^{\sigma(\sit)}\frac{a(\sit)}{b(\sit)}$} for all $\sit \in \sits$. 
In a similar fashion, a net of rationals $\{r_{\sit,n}\}_{\sit \in \sits, n \in \naturalswithzero}$ is called recursive if there are three recursive maps $a,b,\sigma$ from $\sits \times \naturalswithzero \textrm{ to } \naturalswithzero$ such that $b(\sit,n)\neq 0$ for every $\sit \in \sits$ and $n \in \naturalswithzero$, and $r_{\sit,n}=(-1)^{\sigma(\sit,n)}\frac{a(\sit,n)}{b(\sit,n)}$ for all $\sit \in \sits$ and $n \in \naturalswithzero$.

Using these recursive objects, we now move on to define the following mathematical objects that can be computed in an effective manner: computable reals, computable real gambles, computable probability mass functions and, finally, computable real processes such as non-negative supermartingales. 

We say that a sequence $\{r_n\}_{n \in \naturalswithzero}$ of rational numbers converges effectively to a real number $x \in \reals$ if \smash{$\abs{r_n-x} \leq 2^{-N}$} for all $n,N \in \naturalswithzero$ such that $n \geq N$. 
A real number $x$ is then called computable if there is a recursive sequence $\{r_n\}_{n \in \naturalswithzero}$ of rationals that converges effectively to $x$. 
Of course, every rational number is a computable real. 
A gamble $f\colon \posspace \to \reals$ and a probability mass function $p\colon \posspace \to [0,1]$ are computable if $\gamble(x)$ or $p(x)$ is computable for every $x \in \posspace$, respectively. 
After all, finitely many algorithms can be combined into one.

However, a real process $\process \colon \sits \to \reals$ may not be computable even if each of its individual elements $\process(\sit)$ is, with $\sit \in \sits$, because there may be no way to combine the corresponding infinite number of algorithms into one finite algorithm. 
For that reason, we will look at recursive nets of rationals instead of recursive sequences of rationals.
We say that a net $\{r_{\sit,n}\}_{\sit \in \sits, n \in \naturalswithzero}$ of rational numbers converges effectively to a real process $\process\colon \sits \to \reals$ if \smash{$\abs{r_{\sit,n}-\process(\sit)} \leq 2^{-N}$} for all $\sit \in \sits$ and $n,N \in \naturalswithzero$ such that $n \geq N$. 
A real process $\process$ is then called computable if there is a recursive net $\{r_{\sit,n}\}_{\sit \in \sits, n \in \naturalswithzero}$ of rationals that converges effectively to $\process$. 
Of course, every recursive rational process is also a computable real process.
Observe also that, clearly, for any computable real process $F$ and any $\sit \in \sits$, $\process(\sit)$ is a computable real number. 
Furthermore, a constant real process is computable if and only if its constant value is. 

To end this section, we would like to draw attention to the fact that the set of all real processes is uncountable, while the set of all computable real (or recursive rational) processes is countable, simply because the set of all algorithms is countable. 
In the remainder, we will denote by $\callowable{\frcstsystem}$ the set of all computable non-negative supermartingales for the forecasting system $\frcstsystem$.

\section{Computable randomness for forecasting systems}
\label{sec:randomsequences}

At this point, it should be clear how Forecaster's uncertainty about a sequence of variables $X_1,\dots,X_n,\dots$ can be represented by a forecasting system $\frcstsystem$, and that such a forecasting system gives rise to a set of betting strategies whose corresponding capital processes are non-negative (strict) supermartingales.
We will however not allow Sceptic to select any such betting strategy, but will require that her betting strategies should be effectively implementable by requiring that the corresponding non-negative (strict) supermartingales are computable. 
In this way, we restrict Sceptic's betting strategies to a countably infinite set. 
We will now use these strategies to define a notion of computable randomness with respect to a forecasting system $\frcstsystem$. 
The definition uses supermartingales rather than strict supermartingales, but as we will see shortly, this makes no difference.
Loosely speaking, we call a path $\pth$ computably random for $\frcstsystem$ if there is no corresponding computable betting strategy $\supermartin$ that allows Sceptic to become rich without bounds along $\pth$, i.e., $\sup_{n \in \naturalswithzero}\supermartin(\pth^n)=+\infty$, without borrowing.

\begin{definition}
\label{def:randomsequences}
A path $\pth$ is \emph{computably random} for a forecasting system $\frcstsystem$ if there is no computable non-negative real supermartingale \smash{$\supermartin \in \callowable{\frcstsystem}$} that is unbounded along $\pth$. 
We denote the collection of all forecasting systems for which $\pth$ is computably random by $\frcstsystemscomp$.
\end{definition}

It turns out that our definition is reasonably robust with respect to the particular types of supermartingales that are considered.

\begin{proposition} \label{prop:equivalent}
A path $\pth$ is computably random for a forecasting system $\frcstsystem$ if and only if there is no recursive positive rational strict test supermartingale $\test \in \callowable{\frcstsystem}$ such that $\lim_{n \to \infty}\test(\pth^n)=+\infty$.
\end{proposition}

\noindent As a consequence, whenever we restrict Sceptic's allowed betting strategies to a set that is smaller than the one in Definition~\ref{def:randomsequences}, but larger than the one in Proposition~\ref{prop:equivalent}, we obtain a definition for computably random sequences that is equivalent to Definition~\ref{def:randomsequences}. 
Consequently, it indeed does not matter whether we restrict Sceptic's allowed betting strategies to supermartingales or strict supermartingales.

If we consider binary sequences and restrict Sceptic's betting behaviour to non-negative computable test supermartingales, our definition of computable randomness coincides with the one that was recently introduced by De Cooman and De Bock for binary sequences \cite{CoomanBock2017}. 
The equivalence is not immediate though because the forecasting systems in Ref.~\cite{CoomanBock2017} specify probability intervals rather than coherent lower expectations. 
Nevertheless, it does hold because in the binary case, for every coherent lower expectation, the corresponding closed convex set of probability mass functions on $\posspace=\{0,1\}$---see Section 2---is completely characterised by the associated probability interval for the outcome $1$. 
Furthermore, in the case of binary sequences and stationary, precise, computable forecasting systems, it can also be shown that our definition of computable randomness coincides with the classical notion of computable randomness w.r.t. computable probability mass functions \cite{Rute2012}. 

Next, we inspect some properties of computably random sequences $\pth$ and the set of forecasting systems $\frcstsystemscomp$ for which $\pth$ is computably random. 
We start by establishing that for every forecasting system $\frcstsystem$, there is at least one path $\pth \in \pths$ that is computably random for $\frcstsystem$.

\begin{proposition}\label{prop:existence}
For every forecasting system $\frcstsystem$, there is at least one path $\pth$ such that $\frcstsystem \in \frcstsystemscomp$.
\end{proposition}

\noindent Consider now the vacuous forecasting system $\frcstsystem_{,v}$ defined by \mbox{$\lex_{\sit,v} \coloneq \lex_v$} for every $\sit \in \sits$. Our next result shows that the set of forecasting systems $\frcstsystemscomp$ for which $\pth$ is computably random is always non-empty, as it is guaranteed to contain this vacuous forecasting system.

\begin{proposition}\label{prop:VacuousPhi}
All paths are computably random for the vacuous forecasting system: $\frcstsystem_{,v} \in \frcstsystemscomp$ for all $\pth \in \pths$.
\end{proposition}

\noindent Furthermore, if a path $\pth$ is computably random for a forecasting system $\frcstsystem$, then it is also computably random for every forecasting system that is more conservative.

\begin{proposition}
\label{prop:monotonicity}
If $\pth$ is computably random for a forecasting system $\frcstsystem$, i.e., if $\frcstsystem \in \frcstsystemscomp$, then $\pth$ is also computably random for any forecasting system $\lex'_\bullet$ for which $\lex'_\bullet \leq \frcstsystem$, meaning that $\lex'_\sit(\gamble) \leq \lex_\sit(\gamble)$ for all situations $\sit \in \sits$ and gambles $\gamble \in \gambles$. 
\end{proposition}

The following result establishes an abstract generalisation of frequency stabilisation, on which early notions of randomness---like Church randomness---were focused \cite{ambosspies2000}. 
It states that if we systematically buy a gamble $\gamble$ for its coherent lower expectation $\lex(\gamble)$, then in the long run we will not lose any money. 
The connection with frequency stabilisation will become apparent further on in Section~\ref{sec:stat}, where we present an intuitive corollary that deals with running averages of a gamble $\gamble$ along the infinite sequence $\pth$ and its computable infinite subsequences.

\begin{theorem}
\label{theor:ChurchRandomness}
Consider a computable gamble $\gamble$, a forecasting system $\frcstsystem$ for which $\frcstsystem(\gamble)$ is a computable real process, a path $\pth= (\xvaltolong,\dots) \in \pths$ that is computably random for $\frcstsystem$,  and a recursive selection process $\selection: \sits \to \selections$ for which $\lim_{n \to + \infty} \sum_{k=0}^n \selection(\xvaltok) = + \infty$. 
Then
\[
\liminf_{n \to + \infty} \frac{\sum_{k=0}^{n-1}\selection(\xvaltok) \big[ \gamble(\xvalnextsum)-\lex_{\xvaltok}(\gamble) \big]}{\sum_{k=0}^{n-1}\selection(\xvaltok)} \geq 0. 
\]
\end{theorem}

\section{Computable randomness for lower expectations}
\label{sec:stat}

We now introduce a simplified notion of imprecise computable randomness with respect to a single coherent lower expectation; a direct generalisation of computable randomness with respect to a probability mass function. 
We achieve this by simply constraining our attention to stationary forecasting systems: forecasting systems~$\frcstsystem$ that assign the same lower expectation $\lex$ to each situation $\sit \in \sits$. 
In what follows, we will call $\pth$ computably random for a coherent lower expectation~$\lex$ if it is computably random with respect to the corresponding stationary forecasting system. We denote the set of all coherent lower expectations for which $\pth$ is computably random by $\frcstsystemscompstat$.

Since computable randomness for coherent lower expectations is a special case of computable randomness for forecasting systems, the results we obtained before carry over to this simplified setting. 
First, every coherent lower expectation has at least one computably random path. 

\begin{corollary} \label{cor:existence}
For every coherent lower expectation $\lex$, there is at least one path~$\pth$ such that $\lex \in \frcstsystemscompstat$.
\end{corollary}

\noindent Secondly, $\frcstsystemscompstat$ is non-empty as every path $\pth$ is computably random for the vacuous coherent lower expectation $\lex_v$.

\begin{corollary} \label{cor:VacuousE}
All paths are computably random for the vacuous coherent lower expectation: $\lex_v \in \frcstsystemscompstat$ for all $\pth \in \pths$.
\end{corollary}

\noindent Thirdly, if a path $\pth$ is computably random for a coherent lower expectation $\lex \in \frcstsystemscompstat$, then it is also computably random for any coherent lower expectation $\lex'$ that is more conservative.

\begin{corollary}
\label{cor:monotonicityStationary}
If $\pth$ is computably random for a coherent lower expectation $\lex$, then it is also computably random for any coherent lower expectation $\lex'$ for which $\lex' \leq \lex$, meaning that $\lex'(\gamble) \leq \lex(\gamble)$ for every gamble $\gamble \in \gambles$.
\end{corollary}

\noindent And finally, for coherent lower expectations, Theorem~\ref{theor:ChurchRandomness} turns into a property about running averages. 
In particular, it provides bounds on the limit inferior and superior of the running average of a gamble $\gamble$ along the infinite sequence $\pth$ and its computable infinite subsequences. 
Please note that unlike in Theorem~\ref{theor:ChurchRandomness}, we need not impose computability on the gamble $\gamble$ nor on the real number $\lex(\gamble)$.

\begin{corollary}
\label{cor:ChurchRandomnessStationary}
Consider a path $\pth \allowbreak = \allowbreak (\xvaltolong,\allowbreak \dots) \in \pths$, a coherent lower expectation $\lex \in \frcstsystemscompstat$, a gamble $\gamble$ and a recursive selection process $\selection$ for which $\lim_{n \to + \infty}\sum_{k=0}^n \selection(\xvaltok) = + \infty$. 
Then
\begin{align*}
\lex(\gamble) \leq \liminf_{n \to +\infty} \frac{\sum_{k=0}^{n-1}\selection(\xvaltok) \gamble(\xvalnextsum)}{\sum_{k=0}^{n-1}\selection(\xvaltok)} \leq \limsup_{n \to +\infty} \frac{\sum_{k=0}^{n-1}\selection(\xvaltok) \gamble(\xvalnextsum)}{\sum_{k=0}^{n-1}\selection(\xvaltok)} \leq \uex(\gamble).
\end{align*}
\end{corollary}

When comparing our notion of imprecise computable randomness with the classical precise one, there is a striking difference. 
In the precise case, for a given path $\pth$, there may be no probability mass function $p$ for which $\pth$ is computably random (for example, when the running frequencies do not converge). 
But, if there is such a $p$, then it must be unique (because a running frequency cannot converge to two different numbers).
In the imprecise case, however, according to Corollary~\ref{cor:VacuousE} and~\ref{cor:monotonicityStationary}, every path $\pth$ is computably random for the vacuous coherent lower expectation, and if it is computably random for a coherent lower expectation $\lex$, it is also computably random for any coherent lower expectation $\lex'$ that is more conservative---or less informative---than $\lex$. 
This leads us to wonder whether for every path $\pth$, there is a least conservative---or most informative---coherent lower expectation $\lex_\pth$ such that $\pth$ is computably random for every coherent lower expectation $\lex$ that is more conservative than or equal to $\lex_\pth$, but not for any other. 
Clearly, if such a least conservative lower expectation exists, it must be given by
\[
\lex_\pth(\gamble)\coloneq\sup\{\lex(\gamble):\lex \in \frcstsystemscompstat\} \quad \textrm{for all } \gamble \in \gambles,
\]
which is the supremum value of $\lex(\gamble)$ over all coherent lower expectations $\lex$ for which $\pth$ is computably random. 
The crucial question is whether this $\lex_\pth$ is coherent (\textrm{\ref{axiom:coherence:bounds}} and \textrm{\ref{axiom:coherence:homogeneity}} are immediate, but \textrm{\ref{axiom:coherence:subadditivity}} is not) and whether $\pth$ is computably random with respect to $\lex_\pth$. 
If the answer to both questions is yes, then $\lex_\pth$ is the least conservative coherent lower expectation for which $\pth$ is computably random.

The following example illustrates that there are paths $\pth$ for which this is indeed the case. 
It also serves as a nice illustration of some of the results we have obtained so far.

\begin{example} \label{example}
Consider any set $\{p_0,\dots,p_{M-1}\}$ of \(M\) pairwise different, computable probability mass functions, and any path $\pth$ that is computably random for the non-stationary precise forecasting system $\frcstsystem$, defined by $\lex_\sit \coloneqq \ex_{p_{n\!\mod{M}}}$ for all $n\in\naturalswithzero$ and $\sit=\xvalto\in\sits$; it follows from Proposition~\ref{prop:existence} that there is at least one such path.
Then as we are about to show, $\pth$ is computably random for a coherent lower expectation $\lex'$ if and only if $\lex' \leq \lex$, with $\lex(\gamble)\coloneq\min_{k=0}^{M-1}\ex_{p_k}(\gamble)$ for all $\gamble \in \gambles$. 

The `if'-part follows by recalling Proposition~\ref{prop:monotonicity} and noticing that for all $\sit=\xvalto  \in \sits$ and all $\gamble \in \gambles$: 
\[
\lex'(\gamble) 
\leq \lex(\gamble) 
= \min\{\ex_{p_0}(\gamble),\allowbreak\dots, \allowbreak \ex_{p_{M-1}}(\gamble)\} 
\leq \ex_{p_{n\!\mod{M}}}(\gamble)
=\lex_\sit(\gamble).
\]
For the `only if'-part, consider for every $i \in \{0,\dots,M-1\}$ the selection process $\selection_i \colon \sits \to \{0,1\}$ that assumes the value $\selection_i(\xvalto)=1$ whenever $n\!\mod{m}=i$ and $0$ elsewhere. 
Clearly, these selection processes are recursive and $\lim_{n \to \infty} \allowbreak \sum_{k=0}^n \selection_i(\xvaltolongsum) = + \infty$ along the path $\pth=(\xvaltolong,\dots)$---and any other path, in fact. 
Furthermore, due to the computability of the probability mass functions $p_i$, it follows that $\frcstsystem(\gamble)$ is a computable real process for any computable gamble $\gamble \in \gambles$.
For any computable gamble $\gamble \in \gambles$, it therefore follows that
\[
\lex'(\gamble) \leq \liminf_{n \to \infty}\sum_{k=0}^{n-1} \frac{\gamble(x_{i+kM})}{n} \leq \limsup_{n \to \infty} \sum_{k=0}^{n-1} \frac{\gamble(x_{i+kM})}{n} \leq \ex_{p_i}(\gamble),
\]
where the first and third inequality follow from Corollary~\ref{cor:ChurchRandomnessStationary} and Theorem~\ref{theor:ChurchRandomness}, respectively, and the second inequality is a standard property of limits inferior and superior. 
Since (coherent lower) expectations are continuous with respect to uniform convergence \cite{walley1991}, and since every gamble on a finite set $\posspace$ can be uniformly approximated by computable gambles on $\posspace$, the same result holds for non-computable gambles as well. 
Hence, for any gamble $\gamble \in \gambles$ we find that $\lex'(\gamble) \leq \ex_{p_i}(\gamble)$. As this is true for every $i \in \{0,\dots,M-1\}$, it follows that $\lex'(\gamble) \leq \lex(\gamble)$ for all $\gamble \in \gambles$. 

Hence, $\pth$ is indeed computably random for $\lex'$ if and only if $\lex' \leq \lex$. 
Since $\lex$ is clearly coherent itself, this also implies that $\pth$ is computably random with respect to $\lex$ and---therefore--- that $\lex_\pth=\lex$. So for this particular path $\pth$, $\lex_\pth=\lex$ is the least conservative coherent lower expectation for which $\pth$ is computably random. 
\xqed{$\lozenge$}
\end{example}

However, unfortunately, there are also paths for which this is not the case. 
Indeed, as illustrated in Ref.~\cite{CoomanBock2017}, there is a binary path $\pth$---so with $\posspace=\outcomes$--- that is not computably random for $\lex_\pth$ with \smash{$\lex_\pth(\gamble)\coloneq\frac{1}{2}\sum_{x \in \outcomes}\gamble(x)$} for every gamble $\gamble \in \gambles$.

Interestingly, however, in the binary case, it has also been shown that while $\pth$ may not be computably random with respect to $\lex_\pth$, there are always coherent lower expectations $\lex$ that are infinitely close to $\lex_\pth$ and that do make $\pth$ computably random \cite{CoomanBock2017}.\footnote{This result was established in terms of probability intervals; we paraphrase it in terms of coherent lower expectations, using our terminology and notation.} 
So one could say that $\pth$ is `almost' computably random with respect to $\lex_\pth$. 
Whether a similar result continuous to hold in our more general---not necessarily binary---context is an open problem. We conjecture that the answer is yes.

Proving this conjecture is beyond the scope of the present contribution though. Instead, we will establish a similar result for expectation intervals.

\section{Computable randomness for expectation intervals}
\label{sec:intervals}

As a final specialisation of our notion of computable randomness, we now focus on a single gamble $\gamble$ on $\posspace$ and on expectation intervals $I = \sqgroup{\lex(\gamble),\uex(\gamble)}$ that correspond to lower expectations for which $\pth$ is computably random. 
We will denote the set of all closed intervals $I\subseteq[\min f,\max f]$ by $\mathcal{I}_f$.

\begin{definition} \label{def:interval}
A path $\pth$ is \emph{computably random} for a gamble $\gamble \in \gambles$ and a closed interval $I$ if there is a coherent lower expectation \smash{$\lex \in \frcstsystemscompstat$} for which $\lex(\gamble)= \min I$ and $\uex(\gamble)= \max I$. 
For every gamble $\gamble \in \gambles$, we denote the set of all closed intervals for which $\pth$ is computably random by $\mathcal{I}_\gamble(\pth)$.
\end{definition}

\noindent Note that if $\pth$ is computably random for a gamble $\gamble$ and a closed interval $I$, it must be that $I\in\mathcal{I}_f$; so $\mathcal{I}_f(\pth)\subseteq\mathcal{I}_f$. This follows directly from \textrm{\ref{axiom:coherence:bounds}} and conjugacy. 
We can also prove various properties similar to the ones in Section~\ref{sec:randomsequences} and~\ref{sec:stat}. The following result is basically a specialisation of Corollaries~\ref{cor:existence}-\ref{cor:monotonicityStationary}.

\begin{proposition} \label{prop:monotonicityInterval}
Consider any gamble $\gamble \in \gambles$. Then
\begin{enumerate}[label=\emph{(\roman*)}, leftmargin=*]
\item for every $I \in \mathcal{I}_f$, there is at least one $\pth \in \pths$ for which $I \in \mathcal{I}_\gamble(\pth)$; \label{subprop:existence}
\item for every $\pth \in \pths$, $\mathcal{I}_\gamble(\pth)$ is non-empty because $\sqgroup{\min \gamble,\max \gamble} \in \mathcal{I}_\gamble(\pth)$; \label{subprop:non-emptiness}
\item for every $\pth \in \pths$, if $I \in \mathcal{I}_\gamble(\pth)$ and $I \subseteq I'\in\mathcal{I}_f$, then also $I' \in \mathcal{I}_\gamble(\pth)$.\label{subprop:monotonicityforinterval}
\end{enumerate}
\end{proposition}

Moreover, as an immediate consequence of Corollary~\ref{cor:ChurchRandomnessStationary}, if $\pth$ is computably random for a gamble $\gamble$ and a closed interval $I\in\mathcal{I}_f$, then the limit inferior and limit superior of the running averages of the gamble $\gamble$ along the path $\pth$ and its computable infinite subsequences, lie within the interval $I$. 

The properties in Proposition~\ref{prop:monotonicityInterval} lead to a similar question as the one we raised in Section~\ref{sec:stat}, but now for intervals instead of lower expectations. 
Is there, for every path $\pth$ and every gamble $\gamble \in \gambles$, a smallest interval such that $\pth$ is computably random or `almost' computably random for this gamble $\gamble$ and all intervals that contain this smallest interval, but for no other. 
The following result is the key technical step that will allow us to answer this question positively. 
It establishes that when $\pth$ is computably random for a gamble $\gamble$ and two intervals $I_1$ and $I_2$, then it is also computably random for their intersection.

\begin{proposition}\label{prop:intersection}
For any $\pth \in \pths$ and $\gamble \in \gambles$ and for any two closed intervals $I$ and $I'$ in $\mathcal{I}_f$: if $I \in \mathcal{I}_\gamble(\pth)$ and $I' \in \mathcal{I}_\gamble(\pth)$, then $I \cap I' \neq \emptyset$ and $I \cap I' \in \mathcal{I}_\gamble(\pth)$.
\end{proposition}

Together with Proposition~\ref{prop:monotonicityInterval} and the fact that $\mathcal{I}_\gamble(\pth)$ is always non-empty, this result implies that $\mathcal{I}_\gamble(\pth)$ is a filter of closed intervals. 
Since the intersection of a filter of closed intervals in a compact space---such as $\sqgroup{\min \gamble, \max \gamble}$---is always closed and non-empty \cite{aliprantis2006}, it follows that the intersection $\bigcap \mathcal{I}_\gamble(\pth)$ of all closed intervals $I$ for which $\pth$ is computably random with respect to $\gamble$ and $I$, is non-empty and closed, and is therefore a closed interval itself. 
Recalling the discussion in Section~\ref{sec:stat}, it furthermore follows that $\bigcap \mathcal{I}_\gamble(\pth) = \sqgroup{\lex_\pth(\gamble),\uex_\pth(\gamble)}$. 
Similar to what we saw in Section~\ref{sec:stat}, it may or may not be the case that $\pth$ is computably random for the gamble $\gamble$ and the interval $\sqgroup{\lex_\pth(\gamble),\uex_\pth(\gamble)}$; that is, the---possibly infinite---intersection $\bigcap \mathcal{I}_\gamble(\pth)$ may not be an element of $\mathcal{I}_\gamble(\pth)$. 
However, in this interval case, there is a way to completely characterise the models---in this case intervals---for which $\pth$ is computably random. 
To that end, we introduce the following two subsets of $\sqgroup{\min \gamble, \max \gamble}$:
\[
L_\gamble(\pth)\coloneq\{\min I: I \in \mathcal{I}_\gamble(\pth)\} \textrm{ and } U_\gamble(\pth)\coloneq\{\max I: I \in \mathcal{I}_\gamble(\pth)\}.
\]
Due to Proposition~\ref{prop:monotonicityInterval}\emph{\ref{subprop:monotonicityforinterval}}, these sets are intervals: on the one hand $L_\gamble(\pth)=\smash{[\min \gamble,\lex_\pth(\gamble)]}$ or $L_\gamble(\pth)=\smash{[\min \gamble,\lex_\pth(\gamble))}$ and on the other hand $U_\gamble(\pth)=\break\smash{[\uex_\pth(\gamble),\max \gamble]}$ or $U_\gamble(\pth) \allowbreak = \allowbreak \smash{(\uex_\pth(\gamble),\max \gamble]}$. 
As our final result shows, these two intervals allow for a simple characterisation of whether a path $\pth$ is computably random for a gamble $\gamble$ and a closed interval $I$.

\begin{proposition} \label{prop:woehoe}
Consider a path $\pth$, a gamble $\gamble \in \gambles$ and a closed interval $I$. 
Then $I \in \mathcal{I}_\gamble(\pth)$ if and only if $\min I \in L_\gamble(\pth) \textrm{ and } \max I \in U_\gamble(\pth)$.
\end{proposition}

So we see that while $\pth$ may not be computably random for $\gamble$ and the interval \smash{$\sqgroup{\lex_\pth(\gamble),\uex_\pth(\gamble)}$}, it will definitely be `almost' computably random, in the sense that it is surely random for $\gamble$ and any interval $I\in\mathcal{I}_f$ such that \smash{$\min I <\lex_\pth(\gamble)$} and \smash{$\max I > \uex_\pth(\gamble)$}. 
In order to get some further intuition about this result, we consider an example where $L_\gamble(\pth)$ and $U_\gamble(\pth)$ are closed, and where $\pth$ is therefore computably random for $\gamble$ and $\sqgroup{\lex_\pth(\gamble),\uex_\pth(\gamble)}$.

\begin{example}
Consider two probability mass functions $p_0$ and $p_1$, and let the coherent lower expectation $\lex$ be defined by $\lex(\gamble)\coloneq\min\{\ex_{p_0}(\gamble),\ex_{p_1}(\gamble)\}$ for all $\gamble \in \gambles$. 
Then, as we have seen in Example~\ref{example}, there is a path $\pth$ for which $\lex$ is the least conservative coherent lower expectation that makes $\pth$ random. 
Clearly, for any fixed $\gamble \in \gambles$, if we let $I\coloneqq[\lex(f),\uex(f)]$, it follows that $\bigcap \mathcal{I}_\gamble(\pth)=I\in\mathcal{I}_f(\pth)$, and therefore also that $L_\gamble(\pth)=\sqgroup{\min \gamble, \min I}$ and $U_\gamble(\pth)=\sqgroup{\max I,\max \gamble}$. 
Note that in this example, by suitably choosing $p_0$ and $p_1$, $I$ can be any interval in $\mathcal{I}_f$, including the extreme cases where $I=[\min f,\max f]$ or $I$ is a singleton. \xqed{$\lozenge$}
\end{example}

\section{Conclusions and future work}
\label{sec:concl}

We have introduced a notion of computable randomness for infinite sequences that take values in a finite sample space $\posspace$, both with respect to forecasting systems and with respect to two related simpler imprecise uncertainty models: coherent lower expectations and expectation intervals.
In doing so, we have generalised the imprecise notion of computable randomness of De Cooman and De Bock~\cite{CoomanBock2017}, from binary sample spaces to finite ones.

An important observation is that many of their ideas, results and conclusions carry over to our non-binary case. On our account as well as theirs, and in contrast with the classical notion of (precise) computable randomness, every path $\pth$ is for example computably random with respect to at least one uncertainty model, and whenever a path $\pth$ is computably random for a certain uncertainty model, it is also computably random for any uncertainty model that is more conservative---or less informative. 

For many of our results, the move from the binary to the non-binary case was fairly straightforward, and our proofs then mimic those in Ref.~\cite{CoomanBock2017}. For some results, however, additional technical obstacles had to be overcome, all related to the fact that coherent lower expectations are more involved than probability intervals.
Proposition~\ref{prop:intersection}, for example, while similar to an analogous result for probability intervals in Ref.~\cite{CoomanBock2017}, eluded us for quite a while. The key step that made the proof possible is our result that replacing computable (real) betting strategies with recursive (rational) ones leads to an equivalent notion of computable randomness; see Proposition~\ref{prop:equivalent}.

In our future work, we would like to extend our results in Section~\ref{sec:intervals}---that for every path $\pth$ and every gamble $\gamble$, $\pth$ is `almost' computably random for a unique smallest expectation interval---from expectation intervals to coherent lower expectations. 
That is, we would like to prove that every path $\pth$ is `almost' computably random for a unique maximally informative coherent lower expectation. We are convinced that, here too, Proposition~\ref{prop:equivalent} will prove essential.

Furthermore, we would like to develop imprecise generalisations of other classical notions of randomness, such as Martin-L\"of and Schnorr randomness \cite{ambosspies2000}, and explore whether these satisfy similar properties. 
Moreover, we want to explore whether there exist different equivalent imprecise notions of computable randomness in terms of generalised randomness tests, bounded machines etc. \cite{DowneyHirschfeldt2010} instead of supermartingales.
We also wonder if it would be possible to define notions of computable randomness with respect to uncertainty models that are even more general than coherent lower expectations, such as choice functions \cite{jaspergert2019}.

Finally, we believe that our research can function as a point of departure for developing completely new types of imprecise learning methods.
That is, we would like to create and implement novel algorithms that, given a finite sequence of data out of some infinite sequence, estimate the most informative expectation intervals or coherent lower expectation for which the infinite sequence is computably random.
In this way, we obtain statistical methods that are reliable in the sense that they do not insist anymore on associating a single precise probability mass function, which is for example, as was already mentioned in the introduction, not defensible in situations where relative frequencies do not converge.

\subsubsection*{Acknowledgements}
Floris Persiau's research was supported by BOF grant BOF19/DOC/196. 
Jasper De Bock and Gert de Cooman's research was supported by H2020-MSCA-ITN-2016 UTOPIAE, grant agreement 722734.
We would also like to thank the reviewers for carefully reading and commenting on our manuscript.

	 
%
%
%
\bibliographystyle{splncs04}
\bibliography{biblio}

\begin{ArxiveExt}
\section*{Appendix}
\addcontentsline{toc}{section}{Appendix}
\label{app:proofs}
In this Appendix, we have gathered all proofs, and all additional material necessary for understanding the argumentation in these proofs. 

\subsection*{Additional material for Section~\ref{sec:singleforecast}}
\label{app:singleforecast}
In some of our proofs, we will make use of the fact that a coherent lower expectation $\lex$ and its conjugate upper expectation $\uex$ satisfy the following properties \cite[Section 2.6.1]{walley1991}, with $a \in \reals$, $\alpha \geq 0$, $f,g \in \gambles$ and $\{\gamble_n\}_{n \in \naturalswithzero} \in \smash{\gambles^{\naturalswithzero}}$:
\begin{enumerate}[label=\upshape{C}\arabic*.,ref=\upshape{C}\arabic*,leftmargin=*]
\setcounter{enumi}{3}
\item \label{prop:coherence:bounds}{$\min\gamble \leq \lex(\gamble) \leq \uex(\gamble) \leq \max\gamble$}; \upshape\hfill[bounds]
\item \label{prop:coherence:homogeneity}{$\lex(\alpha \gamble) = \alpha \lex(\gamble) \textrm{ and } \uex(\alpha \gamble) = \alpha \uex(\gamble)$}; \upshape\hfill[non-negative homogeneity]
\item \label{prop:coherence:subadditivity}{$\lex(\gamble) + \lex(g) \leq \lex(\gamble+g) \textrm{ and } \uex(\gamble+g) \leq \uex(\gamble) + \uex(g)$}; \upshape\hfill[super/subadditivity]
\item \label{prop:coherence:constantadditivity}{$\lex(\gamble + a) = \lex(\gamble) + a \textrm{ and } \uex(\gamble + a) = \uex(\gamble) + a$}; \upshape\hfill[constant additivity]
\item \label{prop:coherence:increasingness} if $\gamble \leq g $ then $\lex(\gamble) \leq \lex(g) $ and $\uex(\gamble) \leq \uex(g)$; \upshape\hfill[increasingness]
\item \label{prop:coherence:uniformcontinuity} if $\lim_{n \to \infty} \max\abs{\gamble_n-\gamble}=0$ then $\lim_{n \to \infty}\lex(\gamble_n) = \lex(\gamble)$ \hfill \break and $\lim_{n \to \infty}\uex(\gamble_n) = \uex(\gamble)$. \upshape\hfill[uniform continuity]
\end{enumerate}

\subsection*{Proofs and additional material for Section~\ref{sec:forecastingsystems}}
\label{app:forecastingsystems}

In some of our proofs, we will use a particular way of creating test supermartingales.
We define a \emph{multiplier process} as a map $\multprocess$ from $\sits$ to \emph{non-negative} gambles on $\posspace$.
Given such a multiplier process $\multprocess$, we can define a non-negative real process $\mint$ by the recursion equation $\mint(sx)\coloneqq\mint(s)\multprocess(s)(x)$ for all $s\in\sits$ and $x\in\posspace$, with $\mint(\init)\coloneqq1$ and $\sit x$ the concatenation of $\sit$ and $x$.
Any multiplier process $\multprocess$ that satisfies the additional condition that $\uex_{\sit}(\multprocess(s))\leq1$ for all $s\in\sits$, is called a \emph{supermartingale multiplier} for the forecasting system $\frcstsystem$.

\begin{lemma} \label{lemma:supermultitestsuper}
Consider a multiplier process $D$ and a forecasting system $\frcstsystem$. 
If $D$ is a supermartingale multiplier for the forecasting system $\frcstsystem$, then the non-negative real process $\mint$ is a test supermartingale.
\end{lemma}

\begin{proof}{\bf of Lemma~\ref{lemma:supermultitestsuper}\quad}
For every $\sit \in \sits$, since
\begin{align*}
\adddelta\mint(s)(x)
=\mint(sx)-\mint(s)
=\mint(s) \multprocess(s)(x)-\mint(s)
=\mint(s)[\multprocess(s)(x)-1],
\end{align*}
for all $x \in \posspace$, we see that $\adddelta\mint(s) = \mint(s)[\multprocess(s)-1]$ and therefore, that 
\begin{align*}
\uex_{\sit}(\adddelta\mint(s)) 
=\uex_{\sit}(\mint(s)[\multprocess(s)-1])
&\overset{\textrm{\ref{prop:coherence:homogeneity}}}{=}\mint(s)\uex_{\sit}(\multprocess(s)-1) \\
&\overset{\textrm{\ref{prop:coherence:constantadditivity}}}{=}\mint(s)[\uex_{\sit}(\multprocess(s))-1]\leq0,
\end{align*}
where the inequality holds since $\uex_{\sit}(\multprocess(s)) \leq 1$ for all $\sit \in \sits$.
Since $\mint(\init)=1$ by definition, we conclude that $\mint$ is a test supermartingale.
\end{proof}

\subsection*{Proofs and additional material for Section~\ref{sec:computability}}
\label{app:computability}

The following basic results from computability theory are used in some of the proofs of this appendix. 

If $\process$ and $\processtoo$ are computable real processes, then so are $\process+\processtoo$, $\process\processtoo$, $\process/\processtoo$ (provided that $\processtoo(s)\neq0$ for all $s\in\sits$), $\max\set{\process,\processtoo}$, $\min\set{\process,\processtoo}$, $\exp(\process)$, $\ln\process$ (provided that $\process(s)>0$ for all $s\in\sits$), and \smash{$F^{\frac{1}{m}}$} (provided that $\process(s)\geq0$ for all $s\in\sits$) for all $m\in\naturals$ \cite[Chapter~0]{Pour-ElRichards2016}.

For recursive rational processes, similar properties hold. In Ref. \cite{Pour-ElRichards2016}, whenever they want to prove that some operation is recursive, they actually give the algorithm which produces the computation. 
Hence, since finitely many algorithms can be combined into one finite algorithm, it is easy to see that if $F$ and $G$ are recursive rational processes, then so are $\abs{F}$, $\lceil F \rceil$, the rational process $H$ defined by $H(s)=\sum_{k=0}^m F(x_{1:k})$ with $\sit=\xvaltom \in \sits$ and $m \in \naturalswithzero$, the rational process $H$ defined by $\prod_{k=0}^m F(x_{1:k})$ with $\sit=\xvaltom \in \sits$ and $m \in \naturalswithzero$, $F+G$, $FG$, $F/G$ (provided that $G(\sit)\neq 0$ for all $\sit \in \sits$), $\max\{F,G\}$ and $\min\{F,G\}$.
Indeed, consider for example the recursiveness of the rational process $F+G$. 
Since the rational processes $F$ and $G$ are recursive by assumption, there are two algorithms that upon the input of $\sit \in \sits$, output natural numbers $a_1(\sit),a_2(\sit),\sigma_1(\sit),\sigma_2(\sit) \in \naturalswithzero$ and $b_1(\sit),b_2(\sit) \in \naturals$ such that \smash{$F(\sit)=(-1)^{\sigma_1(\sit)}\frac{a_1(\sit)}{b_1(\sit)}$} and \smash{$G(\sit)=(-1)^{\sigma_2(\sit)}\frac{a_2(\sit)}{b_2(\sit)}$}. 
Hence, it remains to prove that there is an algorithm that upon the input of $a_1(\sit),a_2(\sit),\sigma_1(\sit),\sigma_2(\sit),b_1(\sit)$ and $b_2(\sit)$ outputs $a(\sit), \sigma(\sit) \in \naturalswithzero$ and $b(\sit) \in \naturals$ such that \smash{$F(\sit)+G(\sit)=(-1)^{\sigma_1(\sit)}\frac{a_1(\sit)}{b_1(\sit)}+(-1)^{\sigma_2(\sit)}\frac{a_2(\sit)}{b_2(\sit)}=(-1)^{\sigma(\sit)}\frac{a(\sit)}{b(\sit)}$}. 
To that end, we can consider any algorithm that yields 
\begin{flalign*}
a(\sit)&=\abs{(-1)^{\sigma_1(\sit)}a_1(\sit)b_2(\sit)+(-1)^{\sigma_2(\sit)}a_2(\sit)b_1(\sit)}, \\
b(\sit) &=b_1(\sit)b_2(\sit) 
\intertext{and} 
\sigma(\sit) &= \begin{cases}
0 &\textrm{if } (-1)^{\sigma_1(\sit)}a_1(\sit)b_2(\sit)+(-1)^{\sigma_2(\sit)}a_2(\sit)b_1(\sit) \geq 0, \\
1 &\textrm{otherwise.}
\end{cases}
\end{flalign*}
We conclude that, since finitely many algorithms that execute operations on sets of natural numbers and finitely many algorithms that provide those natural numbers can be combined into one finite algorithm, the process $F+G$ is recursive.

Furthermore, since $\posspace$ is a finite set and since finitely many algorithms can be combined into one, we will call a \emph{gamble process} $P \colon \sits \to \gambles$, which is a map from $\sits$ to gambles on $\posspace$, computable (recursive) if for every $x \in \posspace$, the real (rational) process $P(\bullet)(x)$ is computable (recursive).
The two types of gamble processes that we will consider in this appendix are multiplier processes $D$ and process differences $\adddelta F$.

\begin{proposition}\label{prop:computable}
A real process $\process$ is computable if and only if there is a recursive net of rational numbers $\{r_{\sit,n}\}_{\sit \in \sits, n \in \naturalswithzero}$ and a recursive rational function $e\colon \sits \times \naturalswithzero \to \smash{\rationals}$ such that for all $N \in \naturalswithzero$: $n \geq e(s,N) \implies \abs{r_{s,n}-\process(s)}\leq \smash{2^{-N}}$ for all $s\in\sits$ and $n\in\naturalswithzero$.
\end{proposition}

\begin{proof}{\bf of Proposition~\ref{prop:computable}\quad}
For the `only if' part, since $\process$ is computable, there is a recursive net of rational numbers $\{r_{\sit,n}\}_{\sit \in \sits, n \in \naturalswithzero}$ such that $\abs{r_{\sit,n}-\process(\sit)}\leq 2^{-n}$ for all $\sit \in \sits$ and $n \in \naturalswithzero$. It therefore suffices to let $e(\sit,n)\coloneq n$ such that for all $N \in \naturalswithzero \colon n \geq N = e(\sit,N) \implies \abs{r_{\sit,n}-\process(\sit)}\leq 2^{-n} \leq 2^{-N}$ for all $s\in\sits$ and $n\in\naturalswithzero$.

For the `if' part, assume there is a recursive net of rational numbers \break $\{r'_{\sit,n}\}_{\sit \in \sits, n \in \naturalswithzero}$ and a recursive rational function $e\colon\sits\times\naturalswithzero\to\rationals$ such that $n\geq e(s,N)$ implies $\abs{r'_{s,n}-F(s)}\leq2^{-N}$ for all $s\in\sits$ and $n,N\in\naturalswithzero$.
Consider now the net of rational numbers \smash{$\{r_{\sit,n}\}_{\sit \in \sits, n \in \naturalswithzero}$}, defined by \smash{$r_{\sit,n} \coloneqq r'_{s,\max\{0,\lceil e(s,n) \rceil\}}$} for all $\sit \in \sits$ and $n \in \naturalswithzero$. 
This net is recursive because the function $e$---and hence also $\max\{0,\lceil e \rceil \}$---and the net of rational numbers $\{r'_{\sit,n}\}_{\sit \in \sits, n \in \naturalswithzero}$ are both recursive.
Moreover, for every $\sit \in \sits$ and $n,N \in \naturalswithzero$ such that $n \geq N$, since $\max\{0,\lceil e(s,n) \rceil\} \geq \lceil e(s,n) \rceil \geq e(s,n)$ and $\max\{0,\lceil e(s,n) \rceil\} \in \naturalswithzero$, we find that $\abs{r_{s,n}-\process(s)}=\abs{r'_{s,\max\{0,\lceil e(s,n) \rceil\}}-\process(s)}\leq 2^{-n} \leq 2^{-N}$.
Hence, by definition, $\process$ is computable.
\end{proof}
\par
\vspace{\topsep}

We will also require the following result, which establishes a link between the computability of real processes and real multiplier processes.

\begin{proposition}\label{prop:computable:from:multiplier}
For any computable multiplier process $\multprocess$, the associated real process $\mint$ is computable.
\end{proposition}

\begin{proof}{\bf of Proposition~\ref{prop:computable:from:multiplier}\quad}
Assume that the multiplier process $\multprocess$ is computable.
This means that $\multprocess(\bullet)(x)$ is computable for every $x \in \posspace$. Hence, for every $x \in \posspace$, there is a recursive net of rational numbers~$\{r^x_{\sit,n}\}_{\sit \in \sits, n \in \naturalswithzero}$ such that $\abs{D(\sit)(x)-r_{\sit,n}^x} \leq 2^{-n}$ for all $\sit \in \sits$ and $n \in \naturalswithzero$. 
Since $\multprocess$ is non-negative, it follows that $\abs{D(\sit)(x)-\abs{r_{\sit,n}^x}} = \abs{\abs{D(\sit)(x)}-\abs{r_{\sit,n}^x}} \leq \abs{D(\sit)(x)-r_{\sit,n}^x}$ for all $\sit \in \sits$, $n \in \naturalswithzero$ and $x \in \posspace$.
Hence, for every $x \in \posspace$, we may assume without loss of generality that $\{r^x_{\sit,n}\}_{\sit \in \sits, n \in \naturalswithzero}$ is non-negative; if not, we can replace it with the---also---recursive net $\{\abs{r^x_{\sit,n}}\}_{\sit \in \sits, n \in \naturalswithzero}$.
Observe also that $0 \leq D(\sit)(x) \leq 1 + r_{\sit,0}^x$ and $0 \leq r_{\sit,n}^x \leq 2^{-n} + D(\sit)(x) \leq 2^{-n} + 1 + r_{\sit,0}^x$,
for every $\sit \in \sits$, $n \in \naturalswithzero$ and $x \in \posspace$.

Next, we define a recursive net of rational numbers~$\{r_{\sit,n}\}_{\sit \in \sits, n \in \naturalswithzero}$ as follows: for any $\sit=x_{1:m}\in\sits$, with $m\in\naturalswithzero$, and for any $n\in\naturalswithzero$, we let $r_{\sit,n}\coloneqq\prod_{k=0}^{m-1}r_{x_{1:k},n}^{x_{k+1}}$. 
In line with Proposition~\ref{prop:computable}, we will prove that $\mint$ is computable by showing that there is a recursive rational function $e\colon\sits \times \naturalswithzero \to \rationals$ such that for all $N \in \naturalswithzero$: $n \geq e(s,N) \implies \abs{r_{s,n}-\mint(s)}\leq 2^{-N}$ for all $s\in\sits$ and $n\in\naturalswithzero$. 
For any $\sit = x_{1:m} \in \sits$ and $n \in \naturalswithzero$, with $m \in \naturalswithzero$, we see that
\begin{align*}
\abs{\mint(s)-r_{s,n}} 
&= \bigg| \prod_{k=0}^{m-1}D(x_{1:k})(x_{k+1})-\prod_{k=0}^{m-1}r_{x_{1:k},n}^{x_{k+1}} \bigg| \\
&= \bigg| \sum_{l=0}^{m-1}\bigg( \bigg(
\prod_{k=0}^{l-1}r_{x_{1:k},n}^{x_{k+1}}\bigg)
\big(D(x_{1:l})(x_{l+1})-r_{x_{1:l},n}^{x_{l+1}}\big)
\bigg( \prod_{k=l+1}^{m-1}D(x_{1:k})(x_{k+1})
\bigg) \bigg) \bigg| \\
&\leq \sum_{l=0}^{m-1}\bigg(\bigg(
\prod_{k=0}^{l-1}r_{x_{1:k},n}^{x_{k+1}} \bigg)
\abs{D(x_{1:l})(x_{l+1})-r_{x_{1:l},n}^{x_{l+1}}}
\bigg(\prod_{k=l+1}^{m-1}D(x_{1:k})(x_{k+1})\bigg)
\bigg) \\
&\leq 2^{-n}
\sum_{l=0}^{m-1}\bigg(
\prod_{k=0}^{l-1}r_{x_{1:k},n}^{x_{k+1}}
\prod_{k=l+1}^{m-1}D(x_{1:k})(x_{k+1})
\bigg) \\
&\leq 2^{-n}
\sum_{l=0}^{m-1}\bigg(
\prod_{k=0}^{l-1}\Big(2^{-n}+1+r_{x_{1:k},0}^{x_{k+1}}\Big)
\prod_{k=l+1}^{m-1}\Big(1+r_{x_{1:k},0}^{x_{k+1}}\Big)
\bigg) \\
&\leq 2^{-n}
\sum_{l=0}^{m-1}\bigg(
\prod_{k=0}^{l-1}\Big(2+r_{x_{1:k},0}^{x_{k+1}}\Big)
\prod_{k=l+1}^{m-1}\Big(2+r_{x_{1:k},0}^{x_{k+1}}\Big)
\bigg) \\
&\leq 2^{-n}
\sum_{l=0}^{m-1}\bigg(
\prod_{k=0}^{m-1}\Big(2+r_{x_{1:k},0}^{x_{k+1}}\Big)
\bigg) 
= 2^{-n} \alpha(\sit),
\end{align*}
with $\alpha$ the recursive non-negative rational process defined by 
\[
\alpha(\sit) \coloneq \sum_{l=0}^{m-1}\bigg(
\prod_{k=0}^{m-1}\Big(2+r_{x_{1:k},0}^{x_{k+1}}\Big) \bigg),
\] for all $\sit=\xvaltom \in \sits$, with $m \in \naturalswithzero$.
We now consider the recursive rational function $e$ defined by
\begin{align*}
e(\sit,N)
\coloneq N + \alpha(\sit) \quad \textrm{ for all } \sit \in \sits \textrm{ and } N \in \naturalswithzero.
\end{align*}
Then for all $n,N \in \naturalswithzero$ and $\sit \in \sits$, $n \geq e(\sit,N)$ implies that 
\begin{align*}
\abs{\mint(s)-r_{s,n}} 
&\leq 2^{-n} \alpha (\sit)
\leq 2^{-e(\sit,N)} \alpha (\sit)
= 2^{-N}\frac{\alpha (\sit)}{2^{\alpha(\sit)}} 
\leq 2^{-N},
\end{align*}
where the last inequality follows from the fact that $2^y \geq y$ if $y \geq 0$. It therefore follows from Proposition~\ref{prop:computable} that $\mint$ is computable.
\end{proof}

\subsection*{Proofs and additional material for Section~\ref{sec:randomsequences}}
\label{app:randomsequences}

\begin{lemma} \label{lemma:comprec}
For every computable non-negative real supermartingale $\supermartin$, there is a recursive positive rational strict test supermartingale $\supermartin'$  and a positive rational $\alpha>0$ such that $\abs{\alpha \test'(\sit)-\test(\sit)} \leq 7$ for all $\sit \in \sits$.
\end{lemma}

\begin{proof}{\bf of Lemma~\ref{lemma:comprec}\quad}
Consider any computable non-negative real supermartingale $\test$. 
Since $\test$ is computable, there is a recursive net of rational numbers $\{r_{\sit,n}\}_{\sit \in \sits, n \in \naturalswithzero}$ such that 
\begin{equation} \label{hulp:computable}
\abs{\test(\sit)-r_{s,n}}\leq 2^{-n} \quad \textrm{ for every } \sit \in \sits \textrm{ and } n \in \naturalswithzero.
\end{equation} 
For every situation $\sit=\xvaltom \in \sits$, with $m \in \naturalswithzero$, we consider its depth $d(\sit)=m$ in the tree, with of course $d(\sit) \geq d(\init)=0$, and we define the rational process $\test'$ by
\[
\test'(\sit)= \frac{r_{\sit,d(\sit)}+6 (2)^{-d(\sit)}}{r_{\init,0}+6} \quad \textrm{ for all } \sit \in \sits.
\]
Since the function $d$ and the net $\{r_{\sit,n}\}_{\sit \in \sits, n \in \naturalswithzero}$ are recursive, it follows that the rational process $\test'$ is recursive as well. 
Furthermore, it follows from Equation~\eqref{hulp:computable} that $r_{\sit x,d(\sit x)} \leq 2^{-d(\sit x)}+\test(\sit x)$ and $\test(\sit)-2^{-d(\sit)} \leq r_{\sit,d(\sit )}$ for every $\sit \in \sits$ and $x \in \posspace$. 
In particular for $\sit=\init$, this guarantees that $\test'$ is well-defined because $r_{\init,0}+6 \geq M(\init) + 5 >0$, since $M(\init)\geq0$ by assumption.
Moreover, it is easy to see that $\test'(\init)=1$ and that $\test'$ is positive.
It indeed holds that
\[
\test'(\init)= \frac{r_{\init,0}+6}{r_{\init,0}+6} =1
\]
and that
\[
\test'(\sit) = \frac{r_{\sit,d(\sit)}+6(2)^{-d(\sit)}}{r_{\init,0}+6} 
\geq \frac{\test(\sit)+5(2)^{-d(\sit)}}{r_{\init,0}+6} 
\geq \frac{5(2)^{-d(\sit)}}{r_{\init,0}+6} 
>0
\]
for all $\sit \in \sits$.
Next, we show that $\test'$ is a strict supermartingale. For every $\sit \in \sits$, since it holds for every $x \in \posspace$ that
\begin{align*}
\adddelta \test'(\sit)(x)
&=\test'(\sit x)-\test'(\sit) \\
&=\frac{1}{r_{\init,0}+6}(r_{\sit x,d(\sit x)}+6(2)^{-d(\sit x)}-r_{\sit,d(\sit)}-6(2)^{-d(\sit)}) \\
&\leq\frac{1}{r_{\init,0}+6}(\test(\sit x)+7(2)^{-d(\sit x)} - \test(\sit) -5(2)^{-d(\sit)}) \\
&=\frac{1}{r_{\init,0}+6}(\adddelta\test(\sit)(x)+7(2)^{-d(\sit)-1} -10(2)^{-d(\sit)-1}) \\
&=\frac{1}{r_{\init,0}+6}(\adddelta\test(\sit)(x)-3(2)^{-d(\sit)-1}),
\end{align*}
it follows that $\adddelta \test'(\sit) \leq \frac{1}{r_{\init,0}+6}(\adddelta \test(\sit)-3(2)^{-d(\sit)-1})$. For all $\sit \in \sits$, since $\uex_\sit(\adddelta \test(\sit)) \leq 0$ and $r_{\init,0}+6 > 0$, this implies that
\begin{align*}
\uex_\sit(\adddelta \test'(\sit))
&\overset{\textrm{\ref{prop:coherence:increasingness}}}{\leq}\uex_\sit \bigg( \frac{1}{r_{\init,0}+6}\Big(\adddelta\test(\sit)-3(2)^{-d(\sit)-1}\Big)\bigg) \\
&\overset{\textrm{\ref{prop:coherence:homogeneity}}}{=} \frac{1}{r_{\init,0}+6}\uex_\sit\Big(\adddelta\test(\sit)-3(2)^{-d(\sit)-1}\Big) \\
&\overset{\textrm{\ref{prop:coherence:constantadditivity}}}{=} \frac{1}{r_{\init,0}+6}\uex_\sit(\adddelta\test(\sit))-\frac{3(2)^{-d(\sit)-1}}{r_{\init,0}+6} 
\leq -\frac{3(2)^{-d(\sit)-1}}{r_{\init,0}+6} 
< 0.
\end{align*}
We conclude that $\test'$ is a recursive positive rational strict test supermartingale.
For the rational number $\alpha \coloneq r_{\init,0}+6 >0$ and for any $\sit \in \sits$, we furthermore have that
\begin{align*}
\abs{\alpha \test'(\sit)-\test(\sit)}
&=\abs{(r_{\init,0}+6)\test'(\sit)-\test(\sit)} \\
&= \abs{r_{\sit,d(\sit)}+6 (2)^{-d(\sit)}-\test(\sit)}\\
&\leq 6 (2)^{-d(\sit)} + \abs{r_{\sit,d(\sit)}-\test(\sit)} \\
&\leq 6 (2)^{-d(\sit)} + 2^{-d(\sit)} 
= 7(2)^{-d(\sit)} 
\leq 7.
\end{align*}
\end{proof}

\begin{lemma} \label{lemma:limit}
For every recursive positive rational test supermartingale $\test$ and every path $\pth$ such that $\sup_{n \in \naturalswithzero}\test(\pth^n)=+\infty$, there is a computable positive real test supermartingale $\test'$ such that $\lim_{n \to \infty} \test'(\pth^n)=+\infty$.
\end{lemma}

\begin{proof}{\bf of Lemma~\ref{lemma:limit}\quad}
For any $k \in \naturals$, we consider the rational process $\test^{(k)}$, defined for all $\sit =\xvaltom \in \sits$, with $m \in \naturalswithzero$, by
\begin{equation*}
\test^{(k)}(\xvaltom)
\coloneqq
\begin{cases}
2^k&\text{ if $\max_{l \in \{0,\dots,m\}}\test(\xvaltol)\geq 2^k$}, \\
\test(\sit)&\text{ otherwise}.
\end{cases}
\end{equation*}
Clearly, $\test^{(k)}$ is recursive and positive since $\test$ is recursive and positive.
Moreover, since $\test(\init)=1$, we see that also $\test^{(k)}(\init)=1$.
Furthermore, for every $\sit=\xvaltom \in \sits$, with $m \in \naturalswithzero$, since $\uex_\sit(\adddelta \test(\sit))\leq0$, it will follow that also $\uex_\sit(\adddelta \test^{(k)}(\sit))\leq 0$. 
To prove this, we consider two cases.
The first case is when $\test^{(k)}(\sit)=2^k$.
Then for all $x \in \posspace$, it holds that $\test^{(k)}(\sit x)=2^k$.
Hence, $\adddelta\test^{(k)}(\sit)=0$ and therefore $\uex_\sit(\adddelta \test^{(k)}(\sit))=0$ due to \ref{prop:coherence:homogeneity} for $\alpha=0$.
The second case is when $\test^{(k)}(\sit)=\test(\sit) \neq 2^k$. This means that $\max_{l \in \{0,\dots,m\}}\test(\xvaltol) < 2^k$.
Then for all $x \in \posspace$, it follows that $\test^{(k)}(\sit x) = \test(\sit x)$ if $\test(\sit x) < 2 ^k$ and $\test^{(k)}(\sit x) = 2^k \leq \test(\sit x)$ otherwise, and hence in both cases: $\test^{(k)}(\sit x) \leq \test(\sit x)$.
Since $\test^{(k)}(\sit)=\test(\sit)$, this implies that $\adddelta \test^{(k)}(\sit) \leq \adddelta \test(\sit)$ and therefore that $\uex_\sit(\adddelta \test^{(k)}(\sit)) \leq \uex_\sit(\adddelta \test(\sit)) \leq 0$ due to \textrm{\ref{prop:coherence:increasingness}} and since $\test$ is a supermartingale.
In summary, $\test^{(k)}$ is a recursive positive rational test supermartingale.

We now define the recursive net of rational numbers $\{r_{\sit,n}\}_{\sit \in \sits, n \in \naturalswithzero}$ by
\[
r_{\sit,n}\coloneq\sum_{k =1}^n 2^{-k} \test^{(k)}(\sit),
\]
for all $\sit \in \sits$ and $n \in \naturalswithzero$, and consider the process $\test'$ defined by $\test'(\sit)\coloneq \lim_{n \to \infty}r_{\sit,n}$ for all $\sit \in \sits$. 
Since for every $k \in \naturals$ and $\sit=\xvaltom \in \sits$, with $m \in \naturalswithzero$, $\test^{(k)}(\sit)$ is positive and bounded above by $\max_{l \in \{0,\dots,m\}}\test(\xvaltol)$, it follows that $\{r_{\sit,n}\}_{ n \in \naturalswithzero}$ is a non-negative strictly increasing sequence that is also bounded above by $\max_{l \in \{0,\dots,m\}}\test(\xvaltol)$; so $\test'(\sit)$ is a well-defined real positive number. Hence, $\test'$ is a positive real process. 
Next, we consider the recursive rational function $e$ defined by
\begin{align*}
e(\sit,N) 
&\coloneq N+ \max_{l \in \{0,\dots,m\}}\test(\xvaltol),
\end{align*}
for all $N \in \naturalswithzero$ and $\sit=\xvaltom \in \sits$, with $m \in \naturalswithzero$. Then for all $n,N,m \in \naturalswithzero$ and $\sit=\xvaltom \in \sits$, $n \geq e(\sit,N)$ implies that
\begin{align*}
\abs{\test'(\sit)-r_{\sit,n}} 
&=\abs{\test'(\sit) - \sum_{k =1}^n 2^{-k} \test^{(k)}(\sit) } \\
&= \sum_{k =n+1}^{+\infty} 2^{-k} \test^{(k)}(\sit) \\
&\leq \sum_{k =n+1}^{+\infty} 2^{-k} \max_{l \in \{0,\dots,m\}}\test(\xvaltol) \\
&= 2^{-n} \max_{l \in \{0,\dots,m\}}\test(\xvaltol) \\
&\leq 2^{-e(\sit,N)} \max_{l \in \{0,\dots,m\}}\test(\xvaltol) \\
&= 2^{-N} \frac{\max_{l \in \{0,\dots,m\}}\test(\xvaltol)}{2^{\max_{l \in \{0,\dots,m\}}\test(\xvaltol)}} 
\leq 2^{-N},
\end{align*}
where the last inequality follows from the fact that $2^y \geq y$ if $y \geq 0$. Hence, $\test'$ is computable by Proposition~\ref{prop:computable}.

Next, we prove that $\test'$ is a test supermartingale.
Since for all $k \in \naturals$, $\test^{(k)}(\init)=1$, it follows that $\test'(\init)=1$.
Since for every $\sit \in \sits$, $\test'(\sit)= \lim_{n \to \infty} \sum_{k =1}^n \break 2^{-k} \test^{(k)}(\sit)$ is well-defined and real, it follows that for every $\sit \in \sits$ and $x \in \posspace$, 
\begin{align*}
\adddelta \test'(\sit)(x)
=\test'(\sit x)-\test'(\sit) 
&=\lim_{n \to \infty}\sum_{k =1}^n 2^{-k} \test^{(k)}(\sit x) - \lim_{n \to \infty}\sum_{k =1}^n 2^{-k} \test^{(k)}(\sit) \\
&=\lim_{n \to \infty}\sum_{k =1}^n 2^{-k} \adddelta\test^{(k)}(\sit)
\end{align*}
is well-defined and real.
Since $\posspace$ is finite, this implies that \scalebox{0.9}{$\big\{\sum_{k =1}^n 2^{-k} \adddelta \test^{(k)}(\sit)\big\}_{n \in \naturalswithzero}$} converges uniformly to $\adddelta \test'(\sit)$.
Hence, for every $\sit \in \sits$, since the upper expectation $\uex_\sit$ is continuous w.r.t. uniform convergence \eqref{prop:coherence:uniformcontinuity}, it follows that
\begin{align*} 
\uex_\sit(\adddelta\test'(\sit))
=\uex_\sit\Big(\lim_{n \to \infty}\sum_{k =1}^n 2^{-k} \adddelta\test^{(k)}(\sit)\Big)
&\overset{\phantom{\textrm{C}}\textrm{\ref{prop:coherence:uniformcontinuity}}\phantom{\textrm{C}}}{=}\lim_{n \to \infty} \uex_\sit\Big(\sum^n_{k = 1}2^{-k} \adddelta\test^{(k)}(\sit)\Big) \\
&\overset{\textrm{\ref{prop:coherence:homogeneity},\ref{prop:coherence:subadditivity}}}{\leq}\limsup_{n \to \infty} \sum^n_{k = 1}2^{-k} \uex_\sit(\adddelta\test^{(k)}(\sit))
\leq 0,
\end{align*}
where the last inequality follows from the fact that $\uex_\sit(\adddelta\test^{(k)}(\sit)) \leq 0$ for all $k \in \naturals$.
We conclude that $\test'$ is a computable positive real test supermartingale.

Next, we prove that $\lim_{n \to \infty}\test'(\pth^n)=+\infty$ by showing that for every $N \in \naturalswithzero$, there is a $K \in \naturalswithzero$ such that for every $n \geq K$, it holds that $\test'(\pth^n)\geq N$, with $n \in \naturalswithzero$. 
So we fix any $N \in \naturalswithzero$. 
Since $\sup_{n \in \naturalswithzero}\test(\pth^n)=+\infty$, there is a $K \in \naturalswithzero$ such that \smash{$\test(\pth^K) \geq 2^N$}, and hence $\test^{(k)}(\smash{\pth^n})=2^k$ for all $k \in \{1,\dots,N\}$ and $n \in \naturalswithzero$ such that $n \geq K$.
This implies that for all $n \in \naturalswithzero$ such that $n \geq K$, it holds that
\begin{align*}
\test'(\pth^n) 
&= \sum_{k \in \naturals} 2^{-k} \test^{(k)}(\pth^n)
\geq \sum_{k =1}^{N} 2^{-k} \test^{(k)}(\pth^n)
=N,
\end{align*}
where the inequality holds because $\test^{(k)} >0$ for all $k \in \naturals$.
\end{proof}

\begin{proposition} \label{prop:realrat}
Consider any forecasting system $\frcstsystem$. Then for any path $\pth$, the following statements are equivalent:
\begin{enumerate}
\item[\textnormal{(i)}] There is a computable real non-negative supermartingale $\test$ for which \\ $\allowbreak \sup_{n \in \naturalswithzero} \test(\pth^n) \allowbreak= + \infty$.
\item[\textnormal{(ii)}] There is a recursive positive rational strict test supermartingale $\test$ for which $\lim_{n \to \infty} \test(\pth^n)=+ \infty$. 
\end{enumerate}
\end{proposition}

\begin{proof}{\bf of Proposition~\ref{prop:realrat}\quad} (ii)$\Rightarrow$(i) holds trivially.
For the reverse implication, consider any computable real non-negative supermartingale $\test$ such that $\sup_{n \in \naturalswithzero} \test(\pth^n)=+ \infty$. 
By applying Lemma~\ref{lemma:comprec} to $\test$, we know there is a recursive positive rational strict test supermartingale $\test'$ and a positive rational $\alpha>0$ such that $\abs{\alpha \test'(\sit)-\test(\sit)} \leq 7$ for all $\sit \in \sits$.
Since $\sup_{n \in \naturalswithzero}\test(\pth^n)=+\infty$, it follows that $\sup_{n \in \naturalswithzero}\test'(\pth^n)=+\infty$.
By applying Lemma~\ref{lemma:limit} to $\test'$, it therefore follows that there is a computable real positive test supermartingale $\test''$ such that $\lim_{n \to \infty}\test''(\pth^n)=+\infty$.
By again applying Lemma~\ref{lemma:comprec}, this time to $\test''$, we know there is a recursive positive rational strict test supermartingale $\test'''$ and a positive rational $\alpha' >0$ such that $\abs{\alpha' \test'''(\sit)-\test''(\sit)} \leq 7$ for all $\sit \in \sits$.
Since $\lim_{n \to \infty}\test''=+\infty$, it follows that $\lim_{n \to \infty}\test'''=+\infty$, which concludes the proof.
\end{proof}

\par \vspace{\topsep}

\begin{proof}{\bf of Proposition~\ref{prop:equivalent}\quad}
This is an immediate consequence of Proposition~\ref{prop:realrat} and Definition~\ref{def:randomsequences}.
\end{proof}

\par \vspace{\topsep}

\begin{proof}{\bf of Proposition~\ref{prop:existence}\quad}
Let $\test^{(i)}$, $i \in \naturals$, be an enumeration of the countably infinite set of recursive positive rational strict test supermartingales that correspond with the forecasting system $\frcstsystem$, and consider the countably infinite convex combination $\test'\coloneq \sum_{i \in \naturals} 2^{-i}\test^{(i)}$.
For every $i \in \naturals$, $\test^{(i)}$ is positive and $\test^{(i)}(\init)=1$.
Hence, $\test'$ is also positive and $\test'(\init)=1$.
Moreover, for every $\sit \in \sits$, $\test'(\sit)$ is well-defined (but possibly infinite) since it is a countably infinite sum of positive numbers.

Consider now any $\sit \in \sits$ such that $\test'(\sit)$ is real. 
We will show that there is then always an $x \in \posspace$ such that $\test'(\sit x) \leq \test'(\sit)$.
Assume \emph{ex absurdo} that this is not true, i.e., $\test'(\sit)< \test'(\sit x)$ for all $x \in \posspace$.
For every $x \in \posspace$, since $\test'(\sit x)=\sum_{i \in \naturals} 2^{-i} \test^{(i)}(\sit x)$ and $\test'(\sit)$ is real, it follows that there is an $N_x \in \naturals$ such that $\test'(\sit) < \sum_{i=1}^{N_x} 2^{-i} \test^{(i)}(\sit x)$.
Now let $N\coloneq \max_{x \in \posspace} N_x$, which is a well-defined natural number since $\posspace$ is finite.
Then for every $x \in \posspace$, since $\test'\coloneq \sum_{i \in \naturals} 2^{-i}\test^{(i)}$ and each $\test^{(i)}$, $i \in \naturals$, is positive, we see that $\sum_{i=1}^{N} 2^{-i} \test^{(i)}(\sit) < \test'(\sit) < \sum_{i=1}^{N_x} 2^{-i} \test^{(i)}(\sit x) \leq \sum_{i=1}^{N} 2^{-i} \test^{(i)}(\sit x)$, and therefore also $\sum_{i=1}^{N} 2^{-i} \big( \test^{(i)}(\sit x)\break -\test^{(i)}(\sit) \big)>0$.
Since $\posspace$ is finite, it follows that $\min_{x \in \posspace} \big\{ \sum_{i=1}^{N} \big( 2^{-i}\test^{(i)}(\sit x)-2^{-i}\test^{(i)}(\sit) \big)\big\} > 0$. 
However, since for every $i \in \naturals$, it holds that $\uex_\sit(\adddelta\test^{(i)}(\sit)) \leq 0$, we also know that
\begin{align*}
\min_{x \in \posspace} \big\{ \sum_{i=1}^{N} \big( 2^{-i}\test^{(i)}(\sit x)-2^{-i}\test^{(i)}(\sit) \big) \big\} 
&\overset{\phantom{\textrm{\ref{prop:coherence:homogeneity},\ref{prop:coherence:subadditivity}}}}{=} \min \big\{ \sum_{i=1}^{N} 2^{-i}\adddelta\test^{(i)}(\sit) \big\} \\
&\overset{\phantom{\textrm{C}}\textrm{\ref{prop:coherence:bounds}}\phantom{\textrm{C}}}{\leq} \uex_\sit \big( \sum_{i=1}^{N} 2^{-i}\adddelta\test^{(i)}(\sit) \big) \\
&\overset{\textrm{\ref{prop:coherence:homogeneity},\ref{prop:coherence:subadditivity}}}{\leq} \sum_{i=1}^{N} 2^{-i}\uex_\sit(\adddelta\test^{(i)}(\sit)) 
\leq 0,
\end{align*}
which is clearly a contradiction.

Next, we select a path as follows. 
We start at $\init$, with $\test'(\init)=1 \in \reals$. 
As proven above, there is an $x_1 \in \posspace$ such that $ \test'(x_1) \leq \test'(\init)=1$. 
Since $\test'$ is positive, this implies that $\test'(x_1)$ is real.
Since $\test'(x_1)$ is real, we can apply the argument above once more, yielding an $x_2 \in \posspace$ such that $\test'(x_1,x_2) \leq \test'(x_1) \leq \test'(\init)=1$.
By iteration, we obtain a path $\pth=(\xvaltolong,\dots)$ such that $\test'$ is bounded above by $\test'(\init)=1$ along this path.
Now, consider any recursive positive rational strict test supermartingale $\test$.
There is then always an $i' \in \naturals$ such that $\test=\test^{(i')}$.
Since for all $i \in \naturals$, $\test^{(i)}$ is positive, it holds that $\test^{(i')}$ is bounded above by $2^{i'}$ along $\pth$ since for every $n \in \naturalswithzero$:
\[
1 = \test'(\init) \geq \test'(\pth^n) = \sum_{i \in \naturals} 2^{-i} \test^{(i)}(\pth^n) > 2^{-i'} \test^{(i')}(\pth^n). 
\]
Hence, every recursive positive rational strict test supermartingale that corresponds with $\frcstsystem$ remains bounded above along $\pth$. We therefore infer from Proposition~\ref{prop:equivalent} that $\frcstsystem \in \frcstsystemscomp$.
\end{proof}

\par \vspace{\topsep}

\begin{proof}{\bf of Proposition~\ref{prop:VacuousPhi}\quad}
In the imprecise probability tree associated with the vacuous forecasting system $\smash{\lex_{\bullet,v}}$, it holds for every $\sit \in \sits$ and $x \in \posspace$ that $\adddelta \test(\sit)(x) \leq \max(\adddelta \test(\sit)) = \uex_{\sit,v}(\adddelta \test(\sit)) \leq 0$.
Consequently, every computable real non-negative supermartingale $\test$ is bounded above by $\test(\init)$ along any path $\pth \in \pths$ since for every $\sit=\xvalto \in \sits$, with $n \in \naturalswithzero$, it holds that $\test(\sit)=\test(\init) + \sum_{k=0}^{n-1} \adddelta \test(\xvaltok)(x_{k+1}) \leq \test(\init)$.
If we now invoke Definition~\ref{def:randomsequences}, we find that $\smash{\lex_{\bullet,v}} \in \underline{\mathcal{E}}^\sits_\mathrm{C}$ for all $\pth \in \pths$.
\end{proof}

\par \vspace{\topsep}

\begin{proof}{\bf of Proposition~\ref{prop:monotonicity}\quad}
Since $\lex'_\sit(\gamble) \leq \lex_\sit(\gamble)$ for all $\sit \in \sits$ and $\gamble \in \gambles$, it follows from the conjugacy of coherent lower and upper expectations that  $\uex'_\sit(\gamble) \geq \uex_\sit(\gamble)$ for all $\sit \in \sits$ and $\gamble \in \gambles$.
This implies that for all computable non-negative supermartingales $\test \in \smash{\callowabletests{\lex'_\bullet}}$, it holds that $0 \geq \smash{\uex'_\sit(\adddelta \test(\sit))} \geq \uex_\sit(\adddelta \test(\sit))$ for all $\sit \in \sits$, and hence \smash{$\callowabletests{\lex'_\bullet} \subseteq \callowabletests{\frcstsystem}$}. 
From Definition~\ref{def:randomsequences}, we conclude that $\lex'_{\bullet} \in \frcstsystemscomp$, since all computable non-negative supermartingales $\test \in \smash{\callowabletests{\lex_\bullet}}$ remain bounded along $\pth$.
\end{proof}

\begin{lemma} \label{lem:inequality}
For the functions $f,g\colon (-1,+\infty) \to \reals$ defined by $f(y)\coloneq \ln(1+y)$ and $g(y) \coloneq y - y^2$ for all $y \in (-1,+\infty)$, it holds that $f(y) \geq g(y)$ if $y > -\frac{1}{2}$.
\end{lemma}

\begin{proof}{\bf of Lemma~\ref{lem:inequality}\quad}
For $y > -1$, it will turn out useful to know that
\begin{align*}
\frac{df}{dy}(y)-\frac{dg}{dy}(y)
=\frac{1}{1+y} - (1-2y)
=\frac{y(1+2y)}{1+y}.
\end{align*}
If $y \geq 0$, it is easy to see that \smash{$\frac{y(1+2y)}{1+y}\geq 0$}, and hence \smash{$\frac{df}{dy}(y) \geq \frac{dg}{dy}(y)$}.
Similarly, if $-\nicefrac{1}{2}<y \leq 0$, it is easy to see that \smash{$\frac{y(1+2y)}{1+y}\leq 0$}, and hence $\frac{df}{dy}(y) \leq \frac{dg}{dy}(y)$.

Now, for $y=0$, it holds that $f(0)=0=g(0)$. Consequently, for $y > 0$, it holds that $f(y) \geq g(y)$ since $\frac{df(y)}{dy}\geq\frac{dg(y)}{dy}$ for $y \geq 0$.
Similarly, for $y \in (-\nicefrac{1}{2},0)$, it holds that $f(y) \geq g(y)$ since $\frac{df(y)}{dy}\leq\frac{dg(y)}{dy}$ for $y \in (-\nicefrac{1}{2},0]$. 
\end{proof}

\par \vspace{\topsep}
In order to state our next result and prove our next two results, we require an additional notion. Consider a real process $\test \colon\sits\to\reals$ and a selection process $\selection\colon\sits\to\outcomes$. We then define the real process $\avg{\test}_\selection\colon\sits\to\reals$ as follows:
\[
\avg{\test}_\selection(\xvalto) \coloneqq
\begin{cases}
0
&\text{ if $\sum_{k=0}^{n-1}\selection(\xvaltok)=0$}\\
\dfrac{\sum_{k=0}^{n-1}\selection(\xvaltok)
\sqgroup{\adddelta \test(\xvaltok)(x_{k+1})}}
{\sum_{k=0}^{n-1}\selection(\xvaltok)} 
&\text{ if $\sum_{k=0}^{n-1}\selection(\xvaltok)>0$},
\end{cases}
\]
for all $n\in\naturalswithzero$ and $\xvalto\in\sits$. We will now use this notion in the following lemma, whose proof is strongly based on the proof of \cite[Lemma 9]{cooman2015:markovergodic}.

\begin{lemma}\label{lem:martinwlln}
Consider any real $B>0$, $0<\xi<\frac{1}{B}$ and a selection process $\selection$.
Let\/ $\submartin$ be any submartingale such that $\vert\adddelta\submartin\vert\leq B$.
Then the process $\process_\submartin$ defined by:
\begin{equation*}
\process_\submartin(\xvalto)
\coloneqq\prod_{k=0}^{n-1} 
\sqgroup[\big]{1-\xi\selection(\xvaltok)\adddelta\submartin(\xvaltok)(\xvalnextk)} \quad
\text{ for all $n\in\naturalswithzero$ and $\xvalto \in \sits$}
\end{equation*}
is a positive supermartingale with $\process_\submartin(\init)=1$.
Moreover, if in particular $\xi \coloneq \frac{\epsilon}{2B^2}$, with $0<\epsilon<B$, we have that 
\begin{multline*}
\avg{\submartin}_\selection(\xvalto)\leq-\epsilon
\then
\process_\submartin(\xvalto)
\geq\exp\group[\bigg]{\frac{\epsilon^2}{4B^2}\sum_{k=0}^{n-1}\selection(\xvaltok)}\\
\text{for all $n\in\naturalswithzero$ and $\xvalto \in \sits$}.
\end{multline*}
Finally, if $\xi$, $\adddelta\submartin$ and $\selection$ are computable, then $\process_\submartin$ is computable as well.
\end{lemma}

\begin{proof}{\bf of Lemma~\ref{lem:martinwlln}\quad}
$\process_\submartin(\init)=1$ trivially.
To prove that $\process_\submartin$ is a positive supermartingale, consider any $\sit = \xvaltok \in \sits$, with $k \in \naturalswithzero$.
Then for all $x_{k+1} \in \posspace$, it follows from $0<\xi B<1$, $\vert\adddelta\submartin\vert\leq B$ and $\selection(\xvaltok)\in\outcomes$ that $1-\xi\selection(\xvaltok)\adddelta\submartin(\xvaltok)(\xvalnextk)\geq 1-\xi B>0$. 
This tells us that if we let
\begin{equation}\label{eq:multprocess:submartingale}
\multprocess_\submartin(s)\coloneqq1-\xi\selection(s)\adddelta\submartin(s) \quad \textrm{ for all } \sit \in \sits,
\end{equation}
then $\multprocess_\submartin$ is a positive multiplier process.
Moreover, for all $\sit \in \sits$, since $\xi>0$ and $\selection(s)\in\outcomes$, we infer from \textrm{\ref{prop:coherence:homogeneity}}, \textrm{\ref{prop:coherence:constantadditivity}} and the conjugacy of coherent lower and upper expectations that 
\begin{align*}
\uex_\sit(\multprocess_\submartin(\sit))
&=\uex_\sit\group[\big]{1-\xi\selection(s)\adddelta\submartin(s)}\\
&=1+\uex_\sit\group[\big]{-\xi\selection(s)\adddelta\submartin(s)}
=1-\xi\selection(s)\lex_\sit\group[\big]{\adddelta\submartin(s)}\leq1,
\end{align*}
where the inequality holds since $\submartin$ is a submartingale by assumption, and hence $\lex_\sit\group{\adddelta\submartin(s)}\geq0$.
This shows that $\multprocess_\submartin$ is a positive supermartingale multiplier, and therefore $\process_\submartin=\mint[\multprocess_\submartin]$ is indeed a positive supermartingale due to Lemma~\ref{lemma:supermultitestsuper}.

For the second statement, consider any $0<\epsilon<B$ and let $\xi\coloneqq\frac{\epsilon}{2B^2}$. 
Then for any $n\in\naturalswithzero$ and $\xvalto \in \sits$ such that $\avg{\submartin}_\selection(\xvalto)\leq-\epsilon$, we have for all real $K$:
\begin{align}
\process_\submartin(\xvalto)\geq\exp(K) 
&\overset{\phantom{D_M>0}}{\ifandonlyif}
\prod_{k=0}^{n-1} 
\sqgroup[\big]{1-\xi\selection(\xvaltok)\adddelta\submartin(\xvaltok)(\xvalnextk)}\geq\exp(K)\notag\\
&\overset{D_M>0}{\ifandonlyif}
\sum_{k=0}^{n-1} 
\ln\sqgroup[\big]{1-\xi\selection(\xvaltok)\adddelta\submartin(\xvaltok)(\xvalnextk)}\geq K.\label{eq:K}
\end{align}
Since $\vert\adddelta\submartin\vert\leq B$, $\selection(\xvaltok) \in\outcomes$ for all $k \in \{0,\dots,n-1\}$, and $0<\epsilon<B$, we find that
\begin{equation*}
-\xi\selection(\xvaltok)\adddelta\submartin(\xvaltok)
\geq-\xi B
=-\frac{\epsilon}{2B}
>-\frac{1}{2}
\quad\text{ for $0\leq k\leq n-1$}. 
\end{equation*} 
As $\ln(1+y)\ge y-y^2$ for $y>-\frac{1}{2}$ due to Lemma~\ref{lem:inequality}, this allows us to infer that
\begin{multline*}
\sum_{k=0}^{n-1} 
\ln\sqgroup[\big]{1-\xi\selection(\xvaltok)\adddelta\submartin(\xvaltok)(\xvalnextk)}\\
\begin{aligned}
&\geq
\sum_{k=0}^{n-1} 
\sqgroup[\big]{-\xi\selection(\xvaltok)\adddelta\submartin(\xvaltok)(\xvalnextk)
-\xi^2\selection(\xvaltok)^2(\adddelta\submartin(\xvaltok)(\xvalnextk))^2}\\
&=
-\xi\sum_{k=0}^{n-1}\selection(\xvaltok)\avg{\submartin}_\selection(\xvalto)
-\xi^2\sum_{k=0}^{n-1} 
\selection(\xvaltok)(\adddelta\submartin(\xvaltok)(\xvalnextk))^2\\
&\geq\xi\sum_{k=0}^{n-1}\selection(\xvaltok)\epsilon
-\xi^2\sum_{k=0}^{n-1}\selection(\xvaltok)B^2\\
&=\xi(\epsilon-\xi B^2)\sum_{k=0}^{n-1}\selection(\xvaltok)
=\frac{\epsilon^2}{4B^2}\sum_{k=0}^{n-1}\selection(\xvaltok),
\end{aligned}
\end{multline*}
where the first equality holds because $\selection^2=\selection$. 
Now choose $K\coloneqq\frac{\epsilon^2}{4B^2}\sum_{k=0}^{n-1}\selection(\xvaltok)$ in Equation~\eqref{eq:K}.

It remains to prove the last statement, dealing with the computability of $\process_\submartin$.
Since $\adddelta \submartin$ is computable, it holds for every $x \in \posspace$ that $\adddelta \submartin (\bullet)(x)$ is a computable real process. Since $\xi$ and $\selection$ are computable too, we infer from Equation~\eqref{eq:multprocess:submartingale} that for every $x \in \posspace$,  $\multprocess_\submartin(\bullet)(x)=1-\xi\selection(\bullet)\adddelta\submartin(\bullet)(x)$ is a computable real process. Hence, the multiplier process $\multprocess_\submartin$ is computable.
If we now invoke Proposition~\ref{prop:computable:from:multiplier}, we find that $\process_\submartin=\mint[\multprocess_\submartin]$ is indeed computable.  
\end{proof}

\par \vspace{\topsep}

\begin{proof}{\bf of Theorem~\ref{theor:ChurchRandomness}\quad}
Consider a computable gamble $\gamble \in \gambles$ and a forecasting system $\frcstsystem$ for which $\frcstsystem(\gamble)$ is a computable real process.

We define the real process $\submartin$ by
\begin{equation*}
\submartin(\sit)
\coloneqq\sum_{k=1}^{n}\sqgroup[\big]{f(x_k)-\lex_{x_{1:k-1}}(f)},
\end{equation*}
for all $\sit=\xvalto \in\sits$, with $n\in\naturalswithzero$.
Fix any $\sit \in \sits$.
Then
\begin{align}
\adddelta\submartin(\sit)(x)
= \submartin(\sit x)-\submartin(\sit) 
=f(x)-\lex_\sit(f)
\label{eq:lln:local:increments}
\end{align} 
for all $x \in \posspace$, and therefore $\adddelta\submartin(\sit)=f-\lex_\sit(f)$.
Due to \textrm{\ref{prop:coherence:constantadditivity}}, it follows that
\begin{equation*}
\lex_{\sit}(\adddelta\submartin(\sit))
=\lex_{\sit}(f-\lex_\sit(f))
=\lex_{\sit}(f)-\lex_\sit(f)
=0.
\end{equation*}
Since this is true for every $\sit \in \sits$, the real process $\submartin$ is a submartingale.
Furthermore, since the real process $\lex_\bullet(\gamble)$ and the gamble $\gamble$ are computable, we infer from Equation~\eqref{eq:lln:local:increments} that \smash{$\adddelta \submartin$} is computable as well.

Consider now any path $\pth=(\xvaltolong) \in \pths$ that is computably random for $\frcstsystem$, and any recursive selection process $\selection:\sits\to\{0,1\}$ for which $\lim_{n \to + \infty} \sum_{k=0}^n \break \selection(\xvaltok) = + \infty$ and assume {\itshape ex absurdo} that the inequality in Theorem~\ref{theor:ChurchRandomness} is not satisfied:
\begin{equation*}
\liminf_{n \to + \infty} \avg{\submartin}_\selection(\xvalto)
=\liminf_{n \to + \infty} \frac{\sum_{k=0}^{n-1}\selection(\xvaltok) \big[ \gamble(\xvalnextsum)-\lex_{\xvaltok}(\gamble) \big]}{\sum_{k=0}^{n-1}\selection(\xvaltok)} 
 < 0.
\end{equation*}
We will show that there is a computable non-negative supermartingale that is unbounded above along $\pth$, thereby contradicting the assumed computable randomness of $\pth$ for $\frcstsystem$. 
Let\/ $B\coloneqq\max\{1,\varnorm{f}\}>0$, with $\varnorm{f}=\max f-\min f$. 
Then for all $\sit \in \sits$ and $x \in \posspace$, we infer from Equation~\ref{eq:lln:local:increments} that
\begin{align*}
\abs{\adddelta\submartin(\sit)(x)}
\overset{\phantom{\textrm{\ref{prop:coherence:bounds}}}}{=}\abs{f(x)-\lex_\sit(f)}
\leq \max \gamble - \min \gamble = \varnorm{\gamble} \leq B,
\end{align*}
where the inequality follows from the fact that $\min \gamble \leq \gamble(x) \leq \max \gamble$ and (due to \ref{prop:coherence:bounds}) $\min \gamble \leq \lex_\sit(\gamble) \leq \max \gamble$.
Hence, $\abs{\Delta\submartin}\leq B$.
Since $B > 0$ and $\smash{\liminf_{n \to + \infty} \avg{\submartin}_\selection(\xvalto)}<0$, we can choose some computable $\epsilon$ such that $0 < \epsilon < B$ and $\smash{\liminf_{n \to + \infty} \avg{\submartin}_\selection(\xvalto)} < -\epsilon < 0$. 
Since the gamble $f$ is computable, so is $\varnorm{f}=\max \gamble - \min \gamble$ and therefore also $B$. 
Since $\epsilon$ is computable, the real number $\xi \coloneq \frac{\epsilon}{2 B^2}$ is therefore computable as well.

From Lemma~\ref{lem:martinwlln}, we now infer the existence of a computable---positive and hence definitely---non-negative supermartingale $\process_{\submartin}$ such that for all $n \in \naturalswithzero$,
\begin{align} \label{form:notbounded}
\avg{\submartin}_\selection(\xvalto)\leq-\epsilon
\then
\process_\submartin(\xvalto)
\geq\exp\group[\bigg]{\frac{\epsilon^2}{4B^2}\sum_{k=0}^{n-1}\selection(\xvaltok)}.
\end{align}

Since $\pth$ is computably random for $\frcstsystem$, it follows from Definition~\ref{def:randomsequences} that $\sup_{n \in \naturalswithzero}\process_{\submartin}(\xvalto) < + \infty$ and hence, we can fix an $\alpha > 0$ such that $\process_{\submartin}(\xvalto) < \alpha$ for all $n \in \naturalswithzero$.
Since $\lim_{n \to \infty} \sum_{k=0}^n \selection(\xvaltok)=+\infty$, there is some $m \in \naturalswithzero$ such that \smash{$\sum_{k=0}^{m-1}\selection(\xvaltok) \geq \frac{4B^2}{\epsilon^2} \ln(\alpha)$}.
Since $\liminf_{n \to \infty} \avg{\submartin}_\selection(\xvalto) < -\epsilon$, there is some $n \geq m$, with $n \in \naturalswithzero$, such that \smash{$\avg{\submartin}_\selection(\xvalto)\leq-\epsilon$}. 
For this particular $n \in \naturalswithzero$, since $\frac{\epsilon^2}{4B^2} \geq 0$, \smash{$\sum_{k=0}^{n-1}\selection(\xvaltok) \geq \sum_{k=0}^{m-1}\selection(\xvaltok) \geq \frac{4B^2}{\epsilon^2} \ln(\alpha)$} and $\alpha>0$, it follows from Equation~\eqref{form:notbounded} that 
\begin{equation*}
\process_{\submartin}(x_{1:n})
\geq \exp\group[\bigg]{\frac{\epsilon^2}{4B^2}\sum_{k=0}^{n-1}\selection(\xvaltok)} 
\geq  \exp\group[\bigg]{\frac{\epsilon^2}{4B^2} \frac{4B^2}{\epsilon^2} \ln(\alpha)}
\geq \alpha,
\end{equation*}
which clearly contradicts $\process_{\submartin}(\xvalto) < \alpha$.
\end{proof}

\subsection*{Proofs and additional material for Section~\ref{sec:stat}}
\label{app:stat}

In order to state and prove some of the next results, we will use a particular way of constructing coherent lower expectations. For every gamble $\gamble \in \gambles$, every real $\gamma$ such that $\min \gamble \leq \gamma \leq \max \gamble$ and every closed interval $I \in \mathcal{I}_\gamble$ such that $I \subseteq \sqgroup{\min \gamble, \max \gamble}$, consider the operators $\lex_{\gamma,\gamble}$ and $\lex_{I,\gamble}$ defined by
\begin{align} \label{form:Ef}
\lex_{\gamma,\gamble}(g) \coloneq \max\Big\{\min\big(g-\mu(\gamble-\gamma)\big):\mu \geq 0 \Big\} 
\end{align}
and
\begin{align}
\lex_{I,\gamble}(g)&\coloneq\max\big\{\lex_{\min I,\gamble}(g), \lex_{-\max I,-\gamble}(g) \big\} \label{form:Ef-f}
\end{align}
for all $g \in \gambles$, respectively.

\begin{proposition} \label{prop:Ef}
For any gamble $\gamble \in \gambles$ and any real-valued $\gamma \in \reals$ such that $\min \gamble \leq \gamma \leq \max \gamble$, \smash{$\lex_{\gamma,\gamble}$} is a coherent lower expectation that satisfies $\lex_{\gamma,\gamble}(\gamble)=\gamma$ and $\smash{\uex_{\gamma,\gamble}(\gamble)}=\max\gamble$.
Moreover, for every coherent lower expectation $\lex'$ such that $\gamma \leq \lex'(\gamble)$, it holds that $\lex_{\gamma, \gamble} \leq \lex'$. 
\end{proposition} 

\begin{proof}{\bf of Proposition~\ref{prop:Ef}\quad}
To prove that the operator $\lex_{\gamma,\gamble}$ is a coherent lower expectation, we show that it is real-valued and satisfies \textrm{\ref{axiom:coherence:bounds}}--\textrm{\ref{axiom:coherence:subadditivity}}.

To prove \ref{axiom:coherence:bounds}, we
consider any $g \in \gambles$ and show that $\smash{\lex_{\gamma,\gamble}(g)} \geq \min g$. For $\mu = 0$, the right-hand side of Equation \eqref{form:Ef} equals $\min g$. Since we take the maximum for $\mu\geq0$, it follows that $\lex_{\gamma,\gamble}(g) \geq \min g$.

To prove that $\lex_{\gamma,\gamble}$ is a real-valued operator, we consider any $g \in \gambles$ and prove that $\lex_{\gamma,\gamble}(g)$ is real-valued. 
Because of \ref{axiom:coherence:bounds}, we already know that $\lex_{\gamma,\gamble}(g)$ is bounded below by $\min g$. 
To establish the real-valuedness of $\lex_{\gamma,\gamble}(g)$, it therefore suffices to establish a real upper bound as well.
Since $\gamma \leq \max \gamble$ by assumption, we can fix an $x \in \posspace$ such that $\gamble(x) - \gamma \geq 0$. Then it holds that
\begin{align*}
\lex_{\gamma,\gamble}(g)
=\max\Big\{\min\big(g-\mu(\gamble-\gamma)\big):\mu\geq 0 \Big\} 
&\leq \max\big\{g(x)-\mu(\gamble(x)-\gamma):\mu\geq 0 \big\} \\
&= g(x)-0(\gamble(x)-\gamma) 
= g(x).
\end{align*}

To prove \ref{axiom:coherence:homogeneity}, we consider any gamble $g \in \gambles$ and $\alpha\geq0$ and prove that $\smash{\lex_{\gamma,\gamble}(\alpha g)} = \alpha \smash{\lex_{\gamma,\gamble}(g)}$.
If $\alpha = 0$, then
\begin{align*}
\lex_{\gamma,\gamble}(\alpha g)=\lex_{\gamma,\gamble}(0)&=\max\Big\{\min\big(-\mu(\gamble-\gamma)\big):\mu\geq 0 \Big\} \\
&=\max\Big\{\mu\big(\gamma-\max \gamble \big):\mu\geq 0 \Big\}
=0 
= \alpha \lex_{\gamma,\gamble}(g),
\end{align*}
because $\gamma \leq \max \gamble$ by assumption. If $\alpha > 0$, then
\begin{align*}
\lex_{\gamma,\gamble}(\alpha g) &= \max\Big\{\min\big(\alpha g-\mu(\gamble-\gamma)\big):\mu\geq 0 \Big\} \\
&= \max\Big\{\min\big(\alpha g-\alpha \beta(\gamble-\gamma)\big): \beta \geq 0 \Big\} \\
&= \max\Big\{\alpha \min\big( g- \beta(\gamble-\gamma)\big): \beta \geq 0 \Big\} \\
&= \alpha \max\Big\{\min\big( g- \beta(\gamble-\gamma)\big): \beta \geq 0 \Big\} 
=\alpha \lex_{\gamma,\gamble}(g).
\end{align*}

To prove \ref{axiom:coherence:subadditivity}, we consider any $g,h \in \gambles$ and show that $\lex_{\gamma,\gamble}(g+h) \geq \lex_{\gamma,\gamble}(g) + \lex_{\gamma,\gamble}(h)$:
\begin{align*}
\lex_{\gamma,\gamble}(g+h) &= \max\Big\{\min\big(g+h-\mu(\gamble-\gamma)\big):\mu \geq 0 \Big\} \\
&= \max\Big\{\min\big(g+h-(\mu_1+\mu_2)(\gamble-\gamma)\big):\mu_1,\mu_2 \geq 0 \Big\} \\
&= \max\Big\{\min\big(g - \mu_1(\gamble-\gamma)+ h-\mu_2(\gamble-\gamma)\big):\mu_1,\mu_2 \geq 0 \Big\} \\
&\geq \max\Big\{\min\big(g - \mu_1(\gamble-\gamma)\big) + \min\big(h-\mu_2(\gamble-\gamma)\big):\mu_1,\mu_2 \geq 0 \Big\} \\
&= \max\Big\{\min\big(g-\mu_1(\gamble-\gamma)\big):\mu_1\geq 0 \Big\} \\
&\phantom{= \textrm{}}+ \max\Big\{\min\big(h-\mu_2(\gamble-\gamma)\big):\mu_2 \geq 0 \Big\} \\
&=\lex_{\gamma,\gamble}(g)+\lex_{\gamma,\gamble}(h).
\end{align*}
Hence, $\lex_{\gamma, \gamble}$ is indeed a coherent lower expectation.

Next, consider any coherent lower expectation $\lex'$ such that $\gamma \leq \lex'(\gamble)$. For any $g \in \gambles$, it then follows that
\begin{align*}
\lex_{\gamma,\gamble}(g)&\overset{\phantom{P1}}{=}\max\Big\{\min\big(g-\mu(\gamble-\gamma)\big):\mu\geq 0 \Big\} \\
&\overset{\phantom{\textrm{\ref{axiom:coherence:bounds}}}}{=}\max\Big\{\min(g-\mu\gamble)+\mu\gamma:\mu\geq 0 \Big\} \\
&\overset{\textrm{\textrm{\ref{axiom:coherence:bounds}}}}{\leq}\max\Big\{\lex'(g-\mu\gamble)+\mu\lex'(\gamble):\mu\geq 0 \Big\} \\
&\overset{\textrm{\textrm{\ref{axiom:coherence:homogeneity}}}}{=}\max\Big\{\lex'(g-\mu\gamble)+\lex'(\mu\gamble):\mu\geq 0 \Big\} \\
&\overset{\textrm{\textrm{\ref{axiom:coherence:subadditivity}}}}{\leq}\max\Big\{\lex'(g-\mu\gamble+\mu\gamble):\mu\geq 0 \Big\} 
=\lex'(g).
\end{align*}
Hence, \smash{$\lex_{\gamma,\gamble} \leq \lex'$}. 

Finally, we prove that $\lex_{\gamma,\gamble}(\gamble)=\gamma$ and $\uex_{\gamma,\gamble}(\gamble)=\max\gamble$. 
Since $\min \gamble \leq \gamma \leq \max \gamble$, there is a probability mass function $p$ on $\posspace$ such that $\ex_p(\gamble)=\gamma$. 
To show this, we consider two cases: $\min \gamble < \max \gamble$ and $\min \gamble = \max \gamble$. 
If $\min \gamble < \max \gamble$, let $\lambda \coloneq \frac{\max \gamble - \gamma}{\max \gamble - \min \gamble} \in \sqgroup{0,1}$ and $p\coloneq\lambda \mathbb{I}_{\{x_1\}} + (1-\lambda) \mathbb{I}_{\{x_2\}}$, with $x_1,x_2 \in \posspace$ such that $\gamble(x_1)=\min \gamble$ and $\gamble(x_2)=\max \gamble$.
Then it holds that 
\begin{align*}
\ex_p(\gamble)
=\smash{\sum_{x \in \posspace}}p(x)f(x)
&=\lambda \gamble(x_1)+(1-\lambda) \gamble(x_2)\phantom{\sum} \\ 
&=\lambda \min \gamble+(1-\lambda) \max \gamble \\
&= \frac{\max \gamble - \gamma}{\max \gamble - \min \gamble} \min \gamble + \frac{\gamma - \min \gamble}{\max \gamble - \min \gamble} \max \gamble= \gamma.
\end{align*} 
If $\min \gamble = \max \gamble=\gamma$, any probability mass function $p$ on $\posspace$ will do since then $\ex_p(\gamble)=\sum_{x \in \posspace}p(x)f(x)= \gamma \sum_{x \in \posspace}p(x) = \gamma$. 
If we now let $\lex' \coloneq \ex_p$, then $\lex'$ is a coherent lower expectation such that $\lex'(\gamble)=\ex_p(\gamble)=\gamma$, and it therefore follows from the previous part of the proof that \smash{$\lex_{\gamma,\gamble} \leq \lex'$}, with in particular \smash{$\lex_{\gamma,\gamble}(\gamble) \leq \lex'(\gamble) = \ex_p(\gamble)=\gamma$}.
Since also $\lex_{\gamma,\gamble}(\gamble) = \max\{\min(\gamble-\mu(\gamble-\gamma)):\mu \geq 0 \} \geq \min(\gamble-1(\gamble-\gamma)) = \gamma$, it follows that $\lex_{\gamma,\gamble}(\gamble)=\gamma$. 
For \smash{$\uex_{\gamma,\gamble}(\gamble)$}, it holds that
\begin{align*}
\uex_{\gamma,\gamble}(\gamble) 
= -\lex_{\gamma,\gamble}(-\gamble)
&=-\max\Big\{\min\big(-f-\mu(\gamble-\gamma)\big):\mu\geq 0 \Big\} \\
&=\min\Big\{\max\big((1+\mu)f-\mu\gamma\big):\mu\geq 0 \Big\} \\
&=\min\Big\{(1+\mu)\max f-\mu\gamma:\mu\geq 0 \Big\} \\
&=\min\Big\{\max f+\mu\big(\max f -\gamma\big):\mu\geq 0 \Big\} 
= \max f,
\end{align*}
where the last equality holds because the minimum is obtained for $\mu = 0$ since $\max f \geq \gamma$ by assumption.
\end{proof}

\par \vspace{\topsep}

\begin{proof}{\bf of Corollary~\ref{cor:existence}\quad}
This is an immediate consequence of Proposition~\ref{prop:existence} and the definition of computable randomness with respect to a stationary forecasting system.
\end{proof}

\par \vspace{\topsep}

\begin{proof}{\bf of Corollary~\ref{cor:VacuousE}\quad}
This is an immediate consequence of Proposition~\ref{prop:VacuousPhi}, the definition of computable randomness with respect to a stationary forecasting system, and the definition of the vacuous forecasting system $\frcstsystem_{,v}$.
\end{proof}

\par \vspace{\topsep}

\begin{proof}{\bf of Corollary~\ref{cor:monotonicityStationary}\quad}
This is an immediate consequence of Proposition~\ref{prop:monotonicity} and the definition of computable randomness with respect to a stationary forecasting system.
\end{proof}

\par \vspace{\topsep}

\begin{proof}{\bf of Corollary~\ref{cor:ChurchRandomnessStationary}\quad}
Consider a path $\pth \allowbreak = \allowbreak (\xvaltolong,\allowbreak \dots) \in \pths$, a coherent lower expectation $\lex \in \frcstsystemscompstat$, a gamble $\gamble$ and a recursive selection process $\selection$ for which $\lim_{n \to + \infty}\sum_{k=0}^n \selection(\xvaltok) = + \infty$.
Fix any $\epsilon >0$.
Since the computable numbers are dense in the reals, there is a computable gamble $\gamble'$ such that $\gamble - \epsilon \leq \gamble' \leq \gamble$.
Since $\min \gamble'$ is a computable real number and since $\min \gamble' \leq \lex(\gamble') \leq \max \gamble'$ due to \textrm{\ref{prop:coherence:bounds}}, there is a computable real number $\gamma$ such that $\min \gamble' \leq \gamma \leq \lex(\gamble') \leq \max \gamble'$ and $\lex(\gamble')-\epsilon < \gamma$. 
According to Proposition~\ref{prop:Ef}, $\lex_{\gamma,\gamble'}$ is a coherent lower expectation such that $\lex_{\gamma, \gamble'}(\gamble')=\gamma$ and, since $\gamma \leq \lex(\gamble')$, Proposition~\ref{prop:Ef} also implies that $\lex_{\gamma,\gamble'} \leq \lex$. 
Since $\lex_{\gamma,\gamble'} \leq \lex$, it follows from  Corollary~\ref{cor:monotonicityStationary} that since $\pth$ is computably random for $\lex$, it is computably random for $\lex_{\gamma,\gamble'}$ as well. 
Consider now the stationary forecasting system $\lex'_\bullet$ defined by $\lex'_\sit=\lex_{\gamma,\gamble'}$ for all $\sit \in \sits$. 
Since  $\lex_{\gamma,\gamble'}(\gamble')=\gamma$ is a computable real number, $\lex'_\bullet(\gamble')$ is then a computable real process.
Furthermore, since $\pth$ is computably random for $\lex_{\gamma,\gamble'}$, it is also computably random for $\lex'_\bullet$.
Since $\gamble'$ is computable and $\lim_{n \to + \infty}\sum_{k=0}^n \selection(\xvaltok) = + \infty$, it therefore follows from Theorem~\ref{theor:ChurchRandomness} that
\begin{align*}
&\phantom{\ifandonlyif} \textrm{ } \liminf_{n \to + \infty} \frac{\sum_{k=0}^{n-1}\selection(\xvaltok) \big[ \gamble'(\xvalnextsum)-\lex'_{\xvaltok}(\gamble') \big]}{\sum_{k=0}^{n-1}\selection(\xvaltok)} \geq 0 \\
&\ifandonlyif \liminf_{n \to + \infty} \frac{\sum_{k=0}^{n-1}\selection(\xvaltok) \big[ \gamble'(\xvalnextsum)-\lex_{\gamma,\gamble'}(\gamble') \big]}{\sum_{k=0}^{n-1}\selection(\xvaltok)} \geq 0 \\
&\ifandonlyif \liminf_{n \to + \infty} \frac{\sum_{k=0}^{n-1}\selection(\xvaltok) \gamble'(\xvalnextsum)}{\sum_{k=0}^{n-1}\selection(\xvaltok)} -\lex_{\gamma,\gamble'}(\gamble')  \geq 0 \\
&\ifandonlyif
\lex_{\gamma,\gamble'}(\gamble')
\leq \liminf_{n\to+\infty}
\frac{\sum_{k=0}^{n-1}S(\xvaltok)\gamble'(x_{k+1})}{\sum_{k=0}^{n-1}S(\xvaltok)} \leq \liminf_{n\to+\infty}
\frac{\sum_{k=0}^{n-1}S(\xvaltok)\gamble(x_{k+1})}{\sum_{k=0}^{n-1}S(\xvaltok)},
\end{align*}
where the last inequality holds because $\gamble' \leq \gamble$.
By recalling that $\lex_{\gamma,\gamble'}(\gamble')=\gamma$, $\gamma > \lex(\gamble')-\epsilon$ and $\gamble' \geq \gamble - \epsilon$, it is easy to see that
\[
\lex_{\gamma,\gamble'}(\gamble')
=\gamma
> \lex(\gamble')-\epsilon
\overset{\textrm{\ref{prop:coherence:increasingness}}}{\geq} \lex(\gamble - \epsilon) - \epsilon
\overset{\textrm{\ref{prop:coherence:constantadditivity}}}{=}\lex(\gamble)-2\epsilon.
\]
Hence, 
\[
\lex(\gamble)-2\epsilon \leq \liminf_{n\to+\infty}
\frac{\sum_{k=0}^{n-1}S(\xvaltok)\gamble(x_{k+1})}{\sum_{k=0}^{n-1}S(\xvaltok)}.
\]
Since this inequality holds for any $\epsilon >0$, it also holds for $\epsilon = 0$, which establishes the first inequality of the statement. 
The second inequality in the statement of Corollary~\ref{cor:ChurchRandomnessStationary} is a standard property of limits inferior and superior. The third inequality follows from the first inequality and the conjugacy of coherent lower and upper expectations by considering the gamble $-\gamble$:
\begin{align*}
&\phantom{\ifandonlyif} \lex(-\gamble) \leq \liminf_{n\to+\infty}
\frac{\sum_{k=0}^{n-1}S(\xvaltok)\big(-\gamble(x_{k+1})\big)}{\sum_{k=0}^{n-1}S(\xvaltok)}\\
&\ifandonlyif -\uex(\gamble) \leq -\limsup_{n\to+\infty}
\frac{\sum_{k=0}^{n-1}S(\xvaltok)\gamble(x_{k+1})}{\sum_{k=0}^{n-1}S(\xvaltok)}\\
&\ifandonlyif \limsup_{n\to+\infty}
\frac{\sum_{k=0}^{n-1}S(\xvaltok)\gamble(x_{k+1})}{\sum_{k=0}^{n-1}S(\xvaltok)} \leq \uex(\gamble).
\end{align*}
\end{proof}

\subsection*{Proofs and additional material for Section~\ref{sec:intervals}}
\label{app:intervals}

\begin{proposition} \label{prop:Ef-f}
For any gamble $\gamble \in \gambles$ and any interval $I \in \mathcal{I}_\gamble$, \smash{$\lex_{I,\gamble}$} is a coherent lower expectation that satisfies $\smash{\lex_{I,\gamble}}(\gamble)=\min I$ and $\smash{\uex_{I,\gamble}}(\gamble)=\max I$.
Moreover, for every coherent lower expectation $\lex'$ such that $\min I \leq \smash{\lex'}(\gamble)$ and $\smash{\uex'}(\gamble) \leq \max I$, it holds that \smash{$\lex_{I, \gamble} \leq \lex'$}. 
\end{proposition} 

\begin{proof}{\bf of Proposition~\ref{prop:Ef-f}\quad} 
To prove that $\lex_{I,\gamble}$ is a coherent lower expectation, we show that it is real-valued and satisfies \textrm{\ref{axiom:coherence:bounds}}--\textrm{\ref{axiom:coherence:subadditivity}}. Since $\min \gamble \leq \min I \leq \max \gamble$ and $\min (-\gamble) \leq -\max I \leq \max (-\gamble)$ (because $\max \gamble \geq \max I \geq \min \gamble$), it follows from Proposition~\ref{prop:Ef} that \smash{$\lex_{\min I,\gamble}$} and \smash{$\lex_{-\max I,-\gamble}$} are coherent lower expectations. 
By inspecting Equation~\eqref{form:Ef-f}, we then see that $\lex_{I,\gamble}$ is real-valued and satisfies \textrm{\ref{axiom:coherence:bounds}} and \textrm{\ref{axiom:coherence:homogeneity}} trivially because \smash{$\lex_{\min I,\gamble}$} and \smash{$\lex_{-\max I,-\gamble}$} are real-valued and satisfy \textrm{\ref{axiom:coherence:bounds}}--\textrm{\ref{axiom:coherence:homogeneity}}.
To prove that \smash{$\lex_{I,\gamble}$} also satisfies \textrm{\ref{axiom:coherence:subadditivity}}, we consider any two gambles $g,h \in \gambles$ and prove that $\lex_{I,\gamble}(g+h) \geq \lex_{I,\gamble}(g)+\lex_{I,\gamble}(h)$.
If $\lex_{I,\gamble}(g)=\lex_{\min I,\gamble}(g)$ and $\lex_{I,\gamble}(h)=\lex_{\min I,\gamble}(h)$, or $\lex_{I,\gamble}(g)=\lex_{-\max I,-\gamble}(g)$ and $\lex_{I,\gamble}(h) \allowbreak =\lex_{-\max I,-\gamble}(h)$, it follows from Equation~\eqref{form:Ef-f} that
\begin{align*}
\lex_{I,\gamble}(g+h)
&\overset{\phantom{\textrm{\ref{axiom:coherence:subadditivity}}}}{=}\max\big\{\lex_{\min I,\gamble}(g+h), \lex_{-\max I,-\gamble}(g+h) \big\} \\
&\overset{\textrm{\ref{axiom:coherence:subadditivity}}}{\geq}\max\big\{\lex_{\min I,\gamble}(g)+\lex_{\min I,\gamble}(h), \lex_{-\max I,-\gamble}(g)+\lex_{-\max I,-\gamble}(h) \big\} \\
&\overset{\phantom{\textrm{\ref{axiom:coherence:subadditivity}}}}{\geq} \lex_{I,\gamble}(g)+\lex_{I,\gamble}(h).
\end{align*}
Consequently, we may assume without loss of generality that $\smash{\lex_{I,\gamble}\allowbreak(g) \allowbreak} = \break \lex_{\min I,\gamble}(g)$ and $\lex_{I,\gamble}(h)=\lex_{-\max I,-\gamble}(h)$ (if not, just swap the roles of $g$ and $h$). 
Then, due to Equation~\eqref{form:Ef}, we know that there is a $\mu \geq 0$ and a $\lambda \geq 0$ such that \smash{$\lex_{I,\gamble}(g)=\min\big(g-\mu(\gamble-\min I)\big)$} and $\lex_{I,\gamble}(h) \allowbreak=\min\big(h-\lambda(\max I-\gamble)\big)$. 
If $\mu \geq \lambda\geq0$, it follows that
\begin{align*}
\lex_{I,\gamble}(g)+\lex_{I,\gamble}(h)
&=\min\big(g-\mu(\gamble-\min I)\big)+\min\big(h-\lambda(\max I-\gamble)\big) \\
&\leq \min\big(g+h-\mu(\gamble-\min I)-\lambda(\max I-\gamble)\big) \\
&= \min\big(g+h-(\mu-\lambda)(\gamble-\min I)-\lambda(\max I-\min I)\big) \\
&\leq \min\big(g+h-(\mu-\lambda)(\gamble-\min I)\big) \\
&\leq \max\big\{\min(g+h-\alpha(\gamble-\min I)):\alpha \geq 0\big\} \\
&= \lex_{\min I,\gamble}(g+h)
\leq \lex_{I,\gamble}(g+h).
\end{align*}
If $\lambda \geq \mu \geq 0$, a similar argument yields that $\lex_{I,\gamble}(g)+\lex_{I,\gamble}(h)\leq \lex_{-\max I, -\gamble}(g+h) \leq \lex_{I,\gamble}(g+h)$. Hence, $\lex_{I,\gamble}$ is indeed a coherent lower expectation.

Next, we consider any coherent lower expectation $\lex'$ such that $\min I \leq \lex'(\gamble)$ and \smash{$\uex'(\gamble)\leq\max I$}. 
From Proposition~\ref{prop:Ef} and the conjugacy of coherent lower and upper expectations it follows that $\smash{\lex_{\min I,\gamble}} \leq \lex'$ and $\smash{\lex_{-\max I,-\gamble}} \leq \lex'$ since $\lex'(\gamble) \geq \min I$ and $\lex'(-\gamble) = - \uex'(\gamble) \geq -\max I$, respectively. 
Hence, $\lex_{I,\gamble}\leq\lex'$ because, for every $g \in \gambles$, $\allowbreak \smash{\lex_{I,\gamble}}(g)=\max\big\{\lex_{\min I,\gamble}(g),\lex_{-\max I,-\gamble}(g)\big\} \leq \lex'(g)$.

Finally, we prove that $\smash{\lex_{I,\gamble}(\gamble)}=\min I$ and $\uex_{I,\gamble}(\gamble)=\max I$.
Since,\break $\smash{\lex_{\min I,\gamble}}(\gamble)\allowbreak =\min I$, $\smash{\uex_{\min I,\gamble}}(\gamble)=\max \gamble$, $\smash{\lex_{-\max I,-\gamble}}(\gamble)= -\smash{\uex_{-\max I,-\gamble}}(-\gamble)=-\max(-\gamble)=\min\gamble$ and $\smash{\uex_{-\max I,-\gamble}}(\gamble)=\smash{-\lex_{-\max I,-\gamble}(-\gamble)}=-(-\max I)=\max I$ due to Proposition~\ref{prop:Ef} and the conjugacy of coherent lower and upper expectations, it follows that 
\begin{align*}
\lex_{I,\gamble}(\gamble)
= \max\big\{\lex_{\min I, \gamble}(\gamble), \lex_{-\max I , -\gamble}(\gamble)  \big\}
= \max\big\{\min I, \min \gamble \big\}  =\min I 
\end{align*}
and 
\begin{align*}
\uex_{I,\gamble}(\gamble)
=-\lex_{I,\gamble}(-\gamble) 
&= -\max\big\{\smash{\lex_{\min I, \gamble}(-\gamble)}, \smash{\lex_{-\max I , -\gamble}(-\gamble)}  \big\} \\
&=\min\big\{\smash{\uex_{\min I, \gamble}(\gamble)}, \smash{\uex_{-\max I , -\gamble}(\gamble)}  \big\} \\
&= \min\big\{\max \gamble, \max I \big\}
=\max I,
\end{align*} 
since $I \subseteq \sqgroup{\min \gamble, \max \gamble}$ and due to the conjugacy of coherent lower and upper expectations.
\end{proof}

\par
\vspace{\topsep}

\begin{proof}{\bf of Proposition~\ref{prop:monotonicityInterval}\quad}
Consider any gamble $\gamble \in \gambles$ and an interval $I \in \mathcal{I}_\gamble$.
In accordance with Proposition~\ref{prop:Ef-f}, we consider the coherent lower expectation $\lex_{I,\gamble}$ for which $\lex_{I,\gamble}(\gamble)= \min I$ and $\uex_{I,\gamble}(\gamble)= \max I$. 
Due to Corollary~\ref{cor:existence}, there is at least one path $\pth$, such that $\pth$ is computably random for $\lex_{I,\gamble}$ and hence, $I \in \mathcal{I}_\gamble(\pth)$ due to Definition~\ref{def:interval}, which proves statement~\emph{\ref{subprop:existence}}.

Next, consider any path $\pth \in \pths$. 
Due to Corollary~\ref{cor:VacuousE},
$\pth$ is computably random for the vacuous coherent lower expectation $\lex_v$ for which $\lex_v(\gamble)=\min \gamble$ and $\smash{\uex_v}(\gamble)=-\lex_v(-\gamble)=-\min(-\gamble)=\max\gamble$.
Because of Definition~\ref{def:interval}, this implies that $\sqgroup{\min \gamble, \max \gamble} \in \mathcal{I}_\gamble(\pth)$, which proves statement~\emph{\ref{subprop:non-emptiness}}.

Lastly, consider any path $\pth \in \pths$, any $I \in \mathcal{I}_\gamble(\pth)$ and any $I' \in \mathcal{I}_\gamble(\pth)$ such that $I \subseteq I'$. 
Since $I \in \mathcal{I}_\gamble(\pth)$, it follows from Definition~\ref{def:interval} that there is a coherent lower expectation $\lex$ such that $\pth$ is computably random for $\lex$, $\lex(\gamble)= \min I$ and $\smash{\uex(\gamble)}= \max I$.
Since $I \subseteq I'$, $\min I' \leq \min I =\lex(f)$ and $\smash{\uex(f)}=\max I \leq \max I'$.
Since $\lex$ is a coherent lower expectation, it therefore follows from Proposition~\ref{prop:Ef-f} that $\smash{\lex_{I',\gamble}} \leq \lex$.
Since $\pth$ is computably random for $\lex$, Corollary~\ref{cor:monotonicityStationary} now tells us that $\pth$ is also computably random for $\smash{\lex_{I',\gamble}}$.
Due to Definition~\ref{def:interval} and Proposition~\ref{prop:Ef-f}, it follows that $\sqgroup{\smash{\lex_{I',\gamble}}(\gamble),\smash{\uex_{I',\gamble}}(\gamble)}=\sqgroup{\min I', \max I'}=I' \in \mathcal{I}_\gamble(\pth)$.
\end{proof}

\par
\vspace{\topsep}

Consider a gamble $\gamble \in \gambles$ and an interval $I \in \mathcal{I}_\gamble$. 
For the proof of Proposition~\ref{prop:intersection}, it will be useful to know for any gamble $g \in \gambles$ such that \smash{$\uex_{I,\gamble}(g) \leq 1$}, or equivalently (due to Equation~\eqref{form:Ef-f} and the conjugacy of coherent lower and upper expectations), any $g \in \gambles$ such that $\min\big\{\smash{\uex_{\min I,\gamble}(g)},\smash{\uex_{-\max I,-\gamble}(g)} \big\} \leq 1$, whether $\smash{\uex_{\min I,\gamble}}(g) \leq 1$ and/or \smash{$\uex_{-\max I,-\gamble}(g) \leq 1$}.
To that end, we will consider the set $A_{\gamma,\gamble} \coloneq\big\{x \in \posspace: \gamble(x)-\gamma < 0\big\}$.

\begin{lemma}\label{lemma:hulp}
Consider two gambles $h,g \in \gambles$ and a real $\gamma \in \sqgroup{\min h,\max h}$.
If $g(x)>1$ for some $x \in A_{\gamma,h}^c$, then $\uex_{\gamma, h}(g) > 1$.
\end{lemma} 

\begin{proof}{\bf of Lemma~\ref{lemma:hulp}\quad}
Consider any $x \in \smash{A_{\gamma, h}^c}$ such that $g(x)>1$.
Since $x \in \smash{A_{\gamma, h}^c}$, we know that $h(x)-\gamma \geq 0$.
Furthermore, since $\gamma \in \sqgroup{\min h,\max h}$, it follows from Proposition~\ref{prop:Ef} that $\lex_{\gamma,h}$ is a coherent lower expectation.
It therefore follows from the conjugacy of coherent lower and upper expectations and Equation~\eqref{form:Ef} that
\begin{align*}
\uex_{\gamma, h}(g)
=-\lex_{\gamma, h}(-g)
&=-\max\big\{\min(-g-\mu(h-\gamma)): \mu \geq 0\big\} \\
&=\min\big\{\max(g+\mu(h-\gamma)): \mu \geq 0\big\} \\
&\geq \min \big\{g(x)+\mu(h(x)-\gamma): \mu \geq 0\big\} 
=g(x)
> 1,
\end{align*}
where for the last equality, we use the fact that $h(x)-\gamma \geq 0$.
\end{proof}

\begin{lemma}\label{lemma:A}
Consider two gambles $f,g \in \gambles$ and an interval $I \in \mathcal{I}_\gamble$ such that $\smash{\uex_{I,\gamble}}(g) \leq 1$. If $g(x)>1$ for some $x \in A_{\min I,\gamble}$, then $\uex_{\min I, \gamble}(g) \leq 1$. Otherwise, $\uex_{-\max I, -\gamble}(g) \leq 1$.
\end{lemma}\label{prop}
\newpage

\begin{proof}{\bf of Lemma~\ref{lemma:A}\quad}
We start with proving that $A_{\min I, \gamble} \cap A_{-\max I, -\gamble} = \emptyset$. 
Assume \emph{ex absurdo} that there is an $x \in \posspace$ such that $x \in A_{\min I, \gamble}$ and $x \in A_{-\max I, -\gamble}$, then $f(x)-\min I < 0$ and $\max I -f(x) <0$. 
This leads to a contradiction since $f(x)-\min I + \max I -f(x) = \max I - \min I \geq 0$.
Consequently, $A_{\min I, \gamble} \cap A_{-\max I, -\gamble} = \emptyset$ and therefore $\smash{A_{\min I, \gamble}} \subseteq \smash{A_{-\max I, -\gamble}^c}$.

Now, if $g(x)>1$ for some $x \in \smash{A_{\min I, \gamble}} \subseteq \smash{A_{-\max I, -\gamble}^c}$, it follows from Lemma~\ref{lemma:hulp}, with $\gamma=-\max I$ and $h=- \gamble$, that $\uex_{-\max I, -\gamble}(g) > 1$.
Since $I \in \mathcal{I}_\gamble$, and hence $\min \gamble \leq \min I \leq \max \gamble$ and $\min (-\gamble) \leq -\max I \leq \max (-\gamble)$ (because $\max \gamble \geq \max I \geq \min \gamble$), it follows from Proposition~\ref{prop:Ef} and \ref{prop:Ef-f} that \smash{$\lex_{\min I,\gamble}$}, \smash{$\lex_{-\max I,-\gamble}$} and \smash{$\lex_{I,\gamble}$} are coherent lower expectations. 
By inspecting Equation~\eqref{form:Ef-f} and using the conjugacy of coherent lower and upper expectations, it follows that 
\begin{align}
\uex_{I,\gamble}(g)
=-\smash{\lex_{I,\gamble}}(-g)
&=-\max\big\{\lex_{\min I,\gamble}(-g), \lex_{-\max I,-\gamble}(-g) \big\} \nonumber \\
&=\min\big\{-\lex_{\min I,\gamble}(-g), -\lex_{-\max I,-\gamble}(-g) \big\} \nonumber \\
&=\min\big\{\uex_{\min I,\gamble}(g), \uex_{-\max I,-\gamble}(g) \big\}. \label{form:upperI}
\end{align} 
Since $\uex_{I,\gamble}(g)\leq 1$ by assumption, and since \smash{$\uex_{-\max I, -\gamble}(g) > 1$}, it therefore follows that $\smash{\uex_{\min I, \gamble}}(g) \leq 1$.

If there is no $x \in \smash{A_{\min I, \gamble}}$ such that $g(x)>1$, then there are two possibilities: either $g(x)>1$ for some $x \in \smash{A^c_{\min I, \gamble}}$, or $g \leq 1$.
If $g(x)>1$ for some $x \in \smash{A^c_{\min I, \gamble}}$, it follows from Lemma~\ref{lemma:hulp}, with $\gamma=\min I$ and $h=\gamble$, that $\uex_{\min I, \gamble}(g) > 1$.
Since $\smash{\uex_{I,\gamble}(g)}\leq 1$ by assumption, and since $\smash{\uex_{I,\gamble}}(g)=\min\{\smash{\uex_{\min I, \gamble}(g)},\smash{\uex_{-\max I, -\gamble}(g)}\}$ because of Equation~\eqref{form:upperI}, it therefore follows that $\smash{\uex_{-\max I, -\gamble}}(g) \leq 1$.
If $g \leq 1$, then also $\smash{\uex_{-\max I, -\gamble}}(g) \leq 1$, but now because of \ref{prop:coherence:bounds}.
\end{proof}

\par
\vspace{\topsep}

\begin{proof}{\bf of Proposition~\ref{prop:intersection}\quad}
Consider a path $\pth \in \pths$, a gamble $\gamble \in \gambles$ and two intervals $I,I' \in \mathcal{I}_\gamble(\pth)$. 
Without loss of generality, we may assume that  $\min I \geq \min I'$ (if not, switch the role of $I$ and $I'$). We need to prove that $I \cap I' \neq \emptyset$ and $I \cap I' \in \mathcal{I}_\gamble(\pth)$.
We consider two cases: $\max I \leq \max I'$ and $\max I > \max I'$. If $\max I \leq \max I'$, then $I \cap I'= I \neq \emptyset$ because $\min I \geq \min I'$, so $I \cap I' = I \in \mathcal{I}_\gamble(\pth)$ by assumption. So it remains to consider the case $\max I > \max I'$.

We first show that the intersection $I \cap I'$ is non-empty.
Since $\pth$ is computably random for the gamble $\gamble$ and the intervals $I$ and $I'$, it follows from Definition~\ref{def:interval} that there are two coherent lower expectations $\lex$ and $\lex'$ for which $\pth$ is computably random such that $\lex(\gamble)= \min I$, $\smash{\uex(\gamble)}=\max I$, $\lex'(\gamble)= \min I'$ and $\smash{\uex'(\gamble)}=\max I'$.
It then follows from Corollary~\ref{cor:ChurchRandomnessStationary}, for the recursive selection process $\selection$ defined by $\selection(\sit)\coloneq1$ for all $\sit \in \sits$, that
\begin{align*}
\min I = \lex(\gamble) \leq \liminf_{n \to \infty} \frac{\sum_{k=0}^{n-1}\gamble(\xvalnextsum)}{n} 
\leq \limsup_{n \to \infty} \frac{\sum_{k=0}^{n-1} \gamble(\xvalnextsum)}{n} 
\leq \uex'(\gamble) = \max I'.
\end{align*}
Since $\min I' \leq \min I$ and $\max I' < \max I$, this implies that $I \cap I'=\sqgroup{\min I, \max I'} \allowbreak \neq \emptyset$.

Next, we will show that $I \cap I'=\smash{\sqgroup{\min I, \max I'}} \in \smash{\mathcal{I}_\gamble(\pth)}$. 
To that end, we will consider the lower expectation \smash{$\lex_{I \cap I',\gamble}$} and prove that $\pth$ is computably random for \smash{$\lex_{I \cap I',\gamble}$}. 
Since, according to Proposition~\ref{prop:Ef-f}, \smash{$\lex_{I \cap I',\gamble}$} is a coherent lower expectation that satisfies $\smash{\lex_{I \cap I',\gamble}(\gamble)}=\min{I \cap I'}=\min I$ and $\uex_{I \cap I',\gamble}(\gamble)=\max{I \cap I'}=\max I'$, it then follows from Definition~\ref{def:interval} that $I \cap I' \in \mathcal{I}_\gamble(\pth)$. 
To prove that $\pth$ is computably random for $\smash{\lex_{I \cap I',\gamble}}$, we now consider any recursive positive rational strict test supermartingale $\test \in \smash{\callowabletests{\lex_{I \cap I',\gamble}}}$ and will show that $\test$ remains bounded above along $\pth$. 
It then follows from Proposition~\ref{prop:equivalent} that $\pth$ is computably random for $\smash{\lex_{I \cap I',\gamble}}$.

Since $\test$ is positive and $\test(\init)=1$, it has a unique positive multiplier process $\multprocess$ such that $\test=\mint$.
Moreover, since $\test(\bullet)= \mint(\bullet)$ is a recursive rational positive process, it follows that for every $x \in \posspace$, $D(\bullet)(x)=\smash{\frac{\test(\bullet x)}{\test(\bullet)}}$ is a recursive rational process and hence, $\multprocess$ is a recursive rational multiplier process. 
Furthermore, $\multprocess$ is a supermartingale multiplier since for all $\sit \in \sits$, it holds that $D(\sit)(x)=\smash{\frac{\test(\sit x)}{\test(\sit)}}=\smash{\frac{\test(\sit)+\adddelta \test(\sit)(x)}{\test(\sit)}}=1+\frac{\adddelta \test(\sit)(x)}{\test(\sit)}$ for all $x \in \posspace$, and hence
\[
\uex_{I \cap I',\gamble}(D(\sit))
= \uex_{I \cap I',\gamble}\bigg(1+\frac{\adddelta\test(\sit)}{\test(\sit)}\bigg) 
\overset{\textrm{\ref{prop:coherence:homogeneity},\ref{prop:coherence:constantadditivity}}}{=} 1+\frac{1}{\test(\sit)}\uex_{I \cap I',\gamble}(\adddelta\test(\sit))
\leq 1,
\]
because $\test(\sit) >0$ and $\smash{\uex_{I \cap I',\gamble}(\adddelta \test(\sit))}\leq 0 $.
Consider now the rational process $H$ defined by
\[
H(s) \coloneq \begin{cases}
1 &\textrm{if } D(s)(x)>1 \textrm{ for some } x \in A_{\min I \cap I',\gamble},\\
0 &\textrm{otherwise},
\end{cases}
\]
for every $\sit \in \sits$. 
Since $\multprocess$ is recursive, $H$ is recursive as well.
Now, let $D_1$ and $D_2$ be two maps from situations to gambles on $\posspace$, defined by
\[
D_1(s)\coloneq H(s)D(s) +(1-H(s))
\]
and
\[
D_2(s) \coloneq (1-H(s))D(s)+H(s)
\]
for all $\sit \in \sits$.
For every $\sit \in \sits$, if $H(s)=1$, then $D_1(s)=D(s)$ and $D_2(s)=1$, and if $H(s)=0$, then $D_1(s)=1$ and $D_2(s)=D(s)$.
We will show that $D_1$ and $D_2$ are recursive rational positive supermartingale multipliers for $\lex_{\min I,\gamble}$ and $\lex_{-\max I',-\gamble}$, respectively, which are both coherent lower expectations according to Proposition~\ref{prop:Ef} since $\min I \in \smash{\sqgroup{\min \gamble, \max \gamble}}$ and $-\max I' \in \sqgroup{\min(-\gamble),\max(-\gamble)}$ (because $\max I' \in \sqgroup{\min{\gamble},\max{\gamble}}$).

That $D_1$ and $D_2$ are positive follows from the positivity of $\multprocess$ and $1$.
Since $D$ is a recursive rational supermartingale multiplier process and $H$ is a recursive rational process, it follows that for every $x \in \posspace$, $D_1(\bullet)(x)=H(\bullet)D(\bullet)(x) +(1-H(\bullet))$ and $D_2(\bullet)(x)=(1-H(\bullet))D(\bullet)(x)+H(\bullet)$ are recursive rational processes, and hence, $D_1$ and $D_2$ are recursive rational multiplier processes.

Let us now show that they are supermartingale multipliers for $\smash{\lex_{\min I,\gamble}}$ and $\smash{\lex_{-\max I',-\gamble}}$, respectively.
For any $\sit \in \sits$, we have that $\smash{\uex_{\min I,\gamble}(D_1(\sit))} \leq 1$. 
Indeed, if $H(\sit)=0$, then $D_1(\sit)=1$ and hence, $\smash{\uex_{\min I,\gamble}(D_1(\sit))} = \smash{\uex_{\min I,\gamble}(1)} \leq 1$ due to \ref{prop:coherence:bounds}. 
If $H(\sit)=1$, then $D_1(\sit)=D(\sit)$ and $D(\sit)(x)>1$ for some $x \in A_{\min I \cap I',\gamble}$.
Due to Lemma~\ref{lemma:A} and since $\uex_{I \cap I',\gamble}(D(\sit)) \leq 1$, it then follows that $\uex_{\min I \cap I', \gamble}(D(\sit)) \leq 1$.
Since $\min I \cap I' = \min I$ and $D_1(\sit)=D(\sit)$, this implies that $\smash{\uex_{\min I , \gamble}(D_1(\sit))} \leq 1$.
In a similar way, we can prove that $\uex_{-\max I',-\gamble}(D_2(\sit))\leq1$ for every $\sit \in \sits$.
If $H(\sit)=1$, then $D_2(\sit)=1$ and hence, $\smash{\uex_{-\max I',-\gamble}}(D_2(\sit)) = \smash{\uex_{-\max I',-\gamble}}(1) \leq 1$ due to \ref{prop:coherence:bounds}. 
If $H(\sit)=0$, then $D_2(\sit)=D(\sit)$ and $D(\sit)(x) \leq 1$  for all $x \in A_{\min I \cap I',\gamble}$.
Since $\uex_{I \cap I',\gamble}(D(\sit))\leq 1$ and due to Lemma~\ref{lemma:A}, it then follows that $\smash{\uex_{-\max I \cap I', -\gamble}(D(\sit))} \leq 1$.
Since $\max I \cap I' = \max I'$ and $D_2(\sit)=D(\sit)$, this implies that $\smash{\uex_{-\max I', -\gamble}(D_2(\sit))} \leq 1$.

Hence, $D_1$ and $D_2$ are indeed recursive rational positive supermartingale multipliers for \smash{$\lex_{\min I,\gamble}$} and \smash{$\lex_{-\max I',-\gamble}$}, respectively. 
Due to Lemma~\ref{lemma:supermultitestsuper} and Proposition~\ref{prop:computable:from:multiplier}, and since all recursive rational multiplier processes are also computable real multiplier processes, we find that $\test_1\coloneq\mint_1$ and $\test_2\coloneqq\mint_2$ are computable positive test supermartingales for $\lex_{\min I,\gamble}$ and $\lex_{-\max I',-\gamble}$, respectively, i.e., \smash{$\test_1 \in \callowabletests{\lex_{\min I,\gamble}}$} and $\test_2 \in \smash{\callowabletests{\lex_{-\max I',-\gamble}}}$. 

Now recall the coherent lower expectations $\lex, \lex' \in \frcstsystemscompstat$ that we considered before. 
Since $\lex(\gamble) = \min I$, it follows from Proposition~\ref{prop:Ef} that $\smash{\lex_{\min I,\gamble}} \leq \lex$, and hence, $\smash{\lex_{\min I, \gamble} \in \frcstsystemscompstat}$ due to Corollary~\ref{cor:monotonicityStationary}. 
Similarly, since $\smash{\lex'(-\gamble)}=-\uex'(\gamble)=-\max I'$, it follows from Proposition~\ref{prop:Ef} that $\lex_{-\max I',-\gamble} \leq \lex'$, and hence, $\lex_{-\max I', -\gamble} \in \frcstsystemscompstat$ due to Corollary~\ref{cor:monotonicityStationary}. 
So $\pth$ is computably random for the coherent lower expectations $\lex_{\min I,\gamble}$ and $\lex_{-\max I',-\gamble}$.
Since \smash{$\test_1 \in \callowabletests{\lex_{\min I,\gamble}}$} and $\test_2 \in \smash{\callowabletests{\lex_{-\max I',-\gamble}}}$, it therefore follows from Definition~\ref{def:randomsequences} that $\test_1(\pth^n)$ and $\test_2(\pth^n)$ remain bounded above as $n \to \infty$.

Now observe that for any $\sit \in \sits$, since $H(\sit)=1$ or $H(\sit)=0$, it follows that $D_1(\sit)D_2(\sit)= D(\sit)$. 
Hence, $D_1 D_2 = D$ and therefore $\test=\mint=(D_1 D_2)^\circledcirc =\mint_1 \mint_2 = \test_1 \test_2$. Since both $\test_1$ and $\test_2$ remain bounded above along $\pth$, and since $\test_1$ and $\test_2$ are positive, this implies that $\test$ remains bounded above along $\pth$.
\end{proof}

\par 
\vspace{\topsep}

\begin{proof}{\bf of Proposition~\ref{prop:woehoe}\quad}
Consider a path $\pth$, a gamble $\gamble$ and a closed interval $I$.
Since $L_\gamble(\pth)\coloneq\{\min I: I \in \mathcal{I}_\gamble(\pth)\}$ and $U_\gamble(\pth)\coloneq\{\max I: I \in \mathcal{I}_\gamble(\pth)\}$, the `only if'-part holds by construction. 
For the `if'-part, assume that $\min I \in L_\gamble(\pth)$ and $\max I \in U_\gamble(\pth)$. 
It then follows from the definition of $L_\gamble(\pth)$ and $U_\gamble(\pth)$ that there are two intervals $I_1,I_2 \in \mathcal{I}_\gamble(\pth)$ such that $\min I_1 = \min I$ and $\max I_2 = \max I$.
Due to Proposition~\ref{prop:intersection}, it follows that $I_1 \cap I_2 \neq \emptyset$ and $I_1 \cap I_2 \in \mathcal{I}_\gamble(\pth)$.
Since $I_1 \cap I_2 \neq \emptyset$, it follows that 
\begin{align*}
I_1 \cap I_2 
&=\sqgroup{\max\{\min I_1,\min I_2\},\min\{\max I_1,\max I_2\}} \\
&=\sqgroup{\max\{\min I,\min I_2\},\min\{\max I_1,\max I\}} 
\subseteq 
\sqgroup{\min I, \max I}
= I. 
\end{align*}
Since $I_1 \cap I_2 \in \mathcal{I}_\gamble(\pth)$ and due to Proposition~\ref{prop:monotonicityInterval}\emph{\ref{subprop:monotonicityforinterval}}, this implies that $I \in \mathcal{I}_\gamble(\pth)$. 
\end{proof}
\end{ArxiveExt}
\end{document}